\documentclass[11pt]{amsart}

\usepackage[draft]{changes}

\newcommand{\stkout}[1]{\ifmmode\text{\sout{\ensuremath{#1}}}\else\sout{#1}\fi}
\usepackage{romannum}
\usepackage{mathtools}
\usepackage{comment}

\usepackage{amssymb} 

\usepackage{graphics}
\usepackage{graphicx}

\usepackage{amsmath} 
\usepackage{amsthm}
\usepackage{amsfonts}
\usepackage{xcolor}
\usepackage[english]{babel}
\usepackage[margin=1.0in]{geometry}
\usepackage[colorlinks=true,
        linkcolor=blue]{hyperref}
\usepackage{bbm}
\usepackage{verbatim}
\usepackage{mathrsfs}

\numberwithin{equation}{section}

\newtheorem{prop}{Proposition}
\newtheorem{lemma}[prop]{Lemma}

\newtheorem{thm}[prop]{Theorem}

\numberwithin{prop}{section}

\newtheorem{defn}[prop]{Definition}
\theoremstyle{definition}

\newtheorem{rmk}[prop]{Remark}

\definecolor{c1}{rgb}{0.2,0.4,0.5}
\definecolor{c2}{rgb}{0.1,0.3,0.5}
\definecolor{c3}{rgb}{0.2,0.7,0.5}
\usepackage{tikz}

\newcommand{\oo}[1]{\overline{#1}}

\newcommand{\nab}{\nabla}

\newcommand{\w}{\wedge}

\newcommand{\dbar}{\oo\partial}

\newcommand{\mE}{ \mathcal{E}}  
\newcommand{\suchthat}{\mathrel{}\middle|\mathrel{}}

\DeclareMathOperator{\tr}{tr}

\DeclareMathOperator{\Real}{Re}
\DeclareMathOperator{\Ima}{Im}

\begin{document}
\pagenumbering{arabic}

\title[]{The Dirichlet principle for the complex $k$-Hessian functional}

\begin{abstract} We study the variational structure of the complex $k$-Hessian equation on bounded domain $X\subset \mathbb C^n$ with boundary $M=\partial X$. We prove that the Dirichlet problem $\sigma_k (\partial \bar{\partial} u) =0$ in $X$, and $u=f$ on $M$ is variational and we give an explicit construction of the associated functional $\mE_k(u)$. Moreover we prove $\mE_k(u)$ satisfies the Dirichlet 
principle. In a special case when $k=2$, our constructed functional $\mE_2(u)$ involves the Hermitian mean curvature of the boundary, the notion first introduced and studied by X. Wang \cite{Wang}. Earlier work of J. Case and and the first author of this article \cite{CaseWang} introduced a boundary operator for the (real) $k$-Hessian functional which satisfies the Dirichlet principle. The present paper shows that there is a parallel picture in the complex setting. 
\end{abstract}


\author [Wang]{Yi Wang}
\address{Department of Mathematics, Johns Hopkins University, Baltimore, MA 21218, USA} \email{\href{mailto:ywang@math.jhu.edu}{ywang@math.jhu.edu}}

\author [Xu]{Hang Xu}
\address{Department of Mathematics, University of California San Diego, La Jolla, CA 92093, USA}
\email{\href{mailto:h9xu@ucsd.edu}{h9xu@ucsd.edu}}

\maketitle

\section{Introduction}
\label{sec:intro} 

Let $X\subset\mathbb{C}^n$ be a bounded smooth domain with boundary $M=\partial X$.  The usual Dirichlet principle states that
\begin{equation}
\label{model trace}
-\int_X u\Delta u\, dx + \oint_{M} fu_{\nu} d\mu\geq \oint_{M} f(u_f)_{\nu} d\mu
\end{equation}
for all $f\in C^\infty(M)$ and all $u\in C^\infty(\overline{X})$ such that $u\rvert_M=f$. Here $u_{\nu}$ denotes the derivative of $u$ with respect to the unit outward normal vector $\nu$ along $M$, $u_f$ is the harmonic function in $X$ such that $u_f\rvert_M=f$, and $dx$, $d\mu$ are the volume forms on $X$ and $M$, respectively. A standard density argument implies that the trace $u\mapsto u\rvert_M:=\tr u$ extends to a bounded linear operator $\tr\colon W^{1,2}(\overline{X})\to W^{1/2,2}(M)$, while the extension $f\mapsto u_f=:E(f)$ extends to a bounded linear operator $E\colon W^{1/2,2}(M)\to W^{1,2}(\overline{X})$ such that $\tr\circ E$ is the identity.

The Dirichlet principle is a useful tool in many analytic and geometric problems.  The initial observation is in regard to the Dirichlet-to-Neumann map $f\mapsto(u_f)_{\nu}$, a pseudodifferential operator of principle symbol $(-\Delta)^{1/2}$. When $X=\mathbb{R}_+^n$ is the upper half-space, it is the operator $(-\Delta)^{1/2}$.  Thus \eqref{model trace} relates the energy of the local operator $\Delta$ on $X$ to the energy of the nonlocal Dirichlet-to-Neumann operator, providing a useful tool for establishing estimates of PDEs stated in terms of the latter operator.  This strategy is a key motivation for the approach of Caffarelli and Silvestre \cite{CaSi07} to study fractional powers of the Laplacian. As another example, Escobar \cite{Es88,Es90} proved an analogue of \eqref{model trace} on compact manifolds with boundary and use it to recover a sharp Sobolev-trace inequality when $X=\mathbb{R}_+^n$. 
It thus leads to the embedding $W^{1,2}(\overline{\mathbb{R}_+^n})\subset L^{\frac{2(n-1)}{n-2}}(\mathbb{R}^{n-1})$ when $n\geq3$. 
This work is important to studying the boundary Yamabe problem \cite{Es92}.  By considering weights or higher-order operators, Caffarelli and Silvestre \cite{CaSi07} and R. Yang \cite{Yang13} established analogues of \eqref{model trace} for the energy of fractional powers of the Laplacian. Later, the work of \cite{Case15,CaCh16,ChGo11,ChYa17} established analogues of \eqref{model trace} for the energy of the conformal fractional Laplacian (the GJMS operators). 

In \cite{CaseWang}, Case and Wang established a Dirichlet principle for the fully nonlinear operator $\sigma_k(D^2 u)$, where $D^2u$ denotes the Hessian of $u$ on $\mathbb{R}^n$. Later, they \cite{CaseWang18} also developed this idea to study the $k$-curvature, the $k$-th elementary symmetric function of 
the Schouten tensor, on manifolds for $k=1,2$ or when $g$ is a locally conformally flat metric. The purpose of this article is to study 
if the complex $k$-Hessian energy $\sigma_k(D^{1,1}u)$ on $\mathbb{C}^n$ also satisfies the Dirichlet principle and what functional gives rise to it. 

To present the result, we first introduce some notations. In this paper, $D^{1,1}u$ denotes the complex Hessian of $u$, and the $k$-th elementary symmetric function $\sigma_k(A)$ of a Hermitian matrix $A$ (i.e. $\oo{A}^{\intercal}=A$) is defined by
\begin{equation*}
	\sigma_k(A) := \sum_{i_1<\dotsb<i_k} \lambda_{i_1}\dotsm\lambda_{i_k}
\end{equation*}
for $\lambda_1,\cdots,\lambda_n$ the eigenvalues of $A$.  The complex $k$-Hessian equation with Dirichlet boundary condition
\begin{equation} \label{dirichlet problem}
\begin{cases}
\sigma_k(D^{1,1}u) = F(x,u), & \text{in $X$}, \\
u = f(x), & \text{on $M$}
\end{cases}
\end{equation}
has been well-studied for functions $u$ in the elliptic $k$-cone
\begin{equation}
\label{kcone}
\Gamma_k^+ := \left\{ u\in C^\infty(\oo{X}) \suchthat \sigma_j(D^{1,1}u)>0, 1\leq j\leq k \right\}.
\end{equation}
Note that the existence of a solution to \eqref{dirichlet problem} requires that $M$ is \emph{$(k-1)$-pseudoconvex}; i.e. the Levi form $\mathcal{L}$ of $M$ must satisfy $\sigma_j(\mathcal{L})>0$ for $1\leq j\leq k-1$.
Also, in the degenerate case where $F\geq 0$, one needs to consider solutions $u\in \overline{\Gamma_k^+}$, where the closure of the elliptic $k$-cone \eqref{kcone} is with respect to the $C^{1,1}$-norm in $\oo{X}$. 
\eqref{dirichlet problem} is a fully nonlinear analogue of \eqref{model trace}, and is a generalization of complex Monge-Amp\`ere equations (where $k=n$). 

For real Monge-Amp\`ere equation as well as the $k$-Hessian equations, the existence of a unique classical solution for the Dirichlet problem on a domain in $\mathbb R^n$ was proved by Caffarelli, Nirenberg and Spruck in \cite{Caff85}. 
Many other results related to (real) geometric problems were studied by Urbas \cite{Urbas90}, Guan and Li \cite{GL94}, Guan and Guan \cite{GG02}, Chang, Gursky and Yang \cite{CGY02}, Guan and Ma \cite{GM03}, Guan, Lin and Ma \cite{GLM06}, Guan \cite{Guan07}, and references therein.
With regard to the complex Monge-Amp\`ere equations, the existence of plurisubharmonic solution on strictly pseudoconvex domain was proved by Caffarelli, Kohn, Nirenberg and Spruck \cite{Caff85}. In the meanwhile, existence and regularity of solutions to the degenerate ($F(x,u)\geq 0$) Monge-Amp\`ere equation were due to the work of Krylov \cite{Kr87, Kr89a, Kr89b, Kr90}. And it is generally believed that Krylov's method could also prove the same results for degenerate complex $k$-Hessian equations, though we could not find a good reference. (See also \cite{Zhou13}). There are many important results on complex Hessian equations on $\mathbb C^n$ and K\"ahler manifolds, e.g. \cite{CY80}, \cite{G98}, \cite{K98}, \cite{Li04}. Recently the results of \cite{Lu13, DK14, DL15, DK17, DPZ19} have made a lot of new developments on
degenerate complex Hessian equations on compact K\"ahler manifolds as well.


To further understand equation \eqref{dirichlet problem}, it is natural to study the variational structure of equation \eqref{dirichlet problem}. Namely, we would like to establish a fully nonlinear analogue of \eqref{model trace}. There are two questions that arise here. 
First, is there a functional whose critical points satisfy \eqref{dirichlet problem}? And second, if the answer to the first question is yes, does this functional satisfy the Dirichelt principle? The main result of this paper is that the answers to both of these questions are affirmative. 


Let
\begin{equation*}
\mathcal{C}_f := \left\{ u\in C^\infty(\overline{X}) \suchthat u\rvert_M = f \right\}
\end{equation*}
be the class of functions with fixed trace $f\in C^\infty(M)$. Our contribution is the following Dirichlet's principle for such solutions:
\begin{thm}
	\label{thm:main_thm}
	Fix $k\in\mathbb{N}$ and let $X\subset\mathbb{R}^n$ be a bounded $(k-1)$-pseudoconvex domain with boundary $M=\partial X$. Given $f\in C^\infty(M)$, let
	\[ \mathcal{C}_{f,k} := \left\{ u\in \mathcal{C}_f \suchthat D^{1,1}u\in\Gamma_k^+ \right\} . \]
	There is a functional $\mathcal{E}_{k+1}\colon C^2(X)\cap C^2(M)\cap C^1(\oo{X})\rightarrow \mathbb{R}$ such that every $u\in\overline{\mathcal{C}_{f,k}}$ satisfies
	\begin{equation}
	\label{eqn:k-sobolev-trace}
	\mathcal{E}_{k+1}(u) \geq \mathcal{E}_{k+1}(u_f)
	\end{equation}
	where $u_f$ is the unique solution to the Dirichlet problem
	\begin{equation}\label{eqn:degenerate_dirichlet_problem}
	\begin{cases}
	\sigma_k(D^{1,1}u) = 0, & \text{in $X$}, \\
	u = f, & \text{on $M$} ,
	\end{cases}
	\end{equation}
	and $\overline{\mathcal{C}_{f,k}}$ is the closure of $\mathcal{C}_{f,k}$ with respect to the $C^{1,1}$-norm in $\oo{X}$.
\end{thm}

The way we prove the existence of $\mathcal{E}_k$ is by explicitly constructing it using an induction argument. For any $1\leq k\leq n$, let $\sigma_{k}(\overbrace{\cdot, \cdots, \cdot}^k)$ be the polarization of the $k$-linear map $u\rightarrow \sigma_k(D^{1,1}u)$ and let $L\rvert_{\mathcal{H}}: T^{1,0}M\times T^{0,1}M\rightarrow \mathbb{C}$ be the restriction of the second fundamental form $L$ to $T^{1,0}M\times T^{0,1}M$. Then we 
prove that the functional defined below satisfies the Dirichlet principle \eqref{eqn:k-sobolev-trace}:
\begin{equation}\label{Ek}
\mathcal{E}_{k+1}(u):=-\int_X u\sigma_k(D^{1,1}u)dx+\frac{1}{k^2(k+1)}\sum_{i=2}^{k+1}S_i(u),
\end{equation}
where each $S_i$ is a functional on $C^2(X)\cap C^2(M)\cap C^1(\oo{X})$ defined as
\begin{equation}\label{Si def}
	S_i(u):=(-1)^i\frac{k(k+1)}{2}\binom{k}{i-1}\oint_{M}u u_{\nu}^{i-1}\sigma_{k-1}(\overbrace{L\rvert_{\mathcal{H}},\cdots,L\rvert_{\mathcal{H}}}^{i-2}, \overbrace{D^{1,1}u\rvert_{\mathcal{H}}, \cdots, D^{1,1} u}^{k+1-i}\rvert_{\mathcal{H}}).
\end{equation}
In the above formula, $D^{1,1}u\rvert_{\mathcal{H}}: T^{1,0}M\times T^{0,1}M\rightarrow \mathbb{C}$ is the restriction of the complex Hessian $D^{1,1}u$ to $T^{1,0}M\times T^{0,1}M$.

In particular, when $k=1$, we obtain $\mathcal{E}_2(u)=\frac{1}{2}\oint_{M}u u_{\nu}d\mu$, which recovers \eqref{model trace}. When $k=2$, we can further simplify the functional $\mathcal{E}_3(u)$ into:
\begin{align}
\mathcal{E}_3(u)
=-\int_{X}u\sigma_2(D^{1,1}u)
+\frac{1}{2}\oint_{M} uu_{\nu}\Delta_bu
+\frac{1}{8}\oint_{M}uu_{\nu}^2H_b
-\frac{\sqrt{-1}}{2}\oint_{M}uu_{\nu}L\bigl(\oo{\nab}^{1,0}u-\oo{\nab}^{0,1}u,T\bigr),
\end{align}
where $\Delta_b$ is the sub-Laplacian on $M$ and $H_b$ is the Hermitian mean curvature of $M$ introduced in \cite{Wang}.  Here $\oo{\nab}u$ denotes the gradient vector field of $u$ on $M$. $\oo{\nab}^{1,0}u$ and $\oo{\nab}^{0,1}u$ are the projection of $\oo{\nab}u$ onto $T^{1,0}M$ and $T^{0,1}M$ respectively. Let $\nu$ be the unit outward normal vector along $M$ and $T:=J\nu$ where $J$ is the complex structure on $\mathbb{C}^n$. 

The detailed argument to construct $\mathcal{E}_{k+1}$ can be found in Section \ref{sec: dimension k}. It is deduced from the following technical result.

\begin{prop}
	\label{main prop}
	Fix $k\in \mathbb{N}$ and let $X\subset\mathbb{C}^n$ be a bounded smooth domain with boundary $M=\partial X$. Then there is a multilinear differential operator
	\begin{equation}
	\label{boundary requirement}
	B_k\colon C^{\infty}(\overline{X})^k \to C^{\infty}(M)
	\end{equation}
	such that the multilinear form $L_{k+1}\colon C^{\infty}(\oo{X})^{k+1}\to\mathbb{R}$ defined by
	\begin{equation}
	\label{Lk}
	L_{k+1}(u,w^1,\cdots,w^k) := -\int_X u\,\sigma_k(D^{1,1}w^1,\cdots,D^{1,1}w^k)dx + \oint_{M} u\,B_k(w^1,\cdots,w^k)d\mu
	\end{equation}
	is symmetric, where $\sigma_k(D^{1,1}w^1,\cdots,D^{1,1}w^k)$ is the polarization of the $k$-linear map $w\mapsto\sigma_k(D^{1,1}w)$.
\end{prop}

Let us explain briefly why Theorem \ref{thm:main_thm} follows from Proposition \ref{main prop}.
The energy functional in Theorem \ref{thm:main_thm} is actually defined as $\mathcal{E}_{k+1}(u):=L_{k+1}(u,\cdots,u)$. The fact that \eqref{Lk} defines a symmetric $(k+1)$-linear form implies that if $v\in C^\infty(\overline{X})$ is such that $v\rvert_M=0$, then
\begin{equation*}
	\frac{d^j}{dt^j}\Big\vert_{t=0}\mathcal{E}_k(u+tv) = -\frac{(k+1)!}{(k+1-j)!}\int_X v\,\sigma_k\bigl(\overbrace{D^{1,1}v,\cdots,D^{1,1}v}^{j-1},\overbrace{D^{1,1}u,\cdots,D^{1,1}u}^{k+1-j}\bigr) dx,
\end{equation*}
for all $1\leq j\leq k+1$.  That is, within the class $\mathcal{C}_f$, the derivatives of the energies $\mathcal{E}_k$ depend only on the interior integrals.  In particular, it is straightforward to identify the critical points of $\mathcal{E}_k$ and deduce the convexity of $\mathcal{E}_k$ within the positive cone $\Gamma_k^+$, from which Theorem \ref{thm:main_thm} follows readily.

Given $u\in C^2(X)\cap C^2(M)\cap C^1(\oo{X})$ and $k\in \mathbb{N}$, we set 
\begin{equation}\label{Qk}
	Q_k(u):=\frac{1}{2k}\sum_{i=2}^{k+1}(-1)^{i}\binom{k}{i-1}u_{\nu}^{i-1}\sigma_{k-1}(\overbrace{L\rvert_{\mathcal{H}},\cdots,L\rvert_{\mathcal{H}}}^{i-2}, \overbrace{D^{1,1}u\rvert_{\mathcal{H}}, \cdots, D^{1,1} u\rvert_{\mathcal{H}}}^{k+1-i}).
\end{equation}
\eqref{Ek} and \eqref{Si def} imply that
\begin{equation*}
	\mathcal{E}_{k+1}(u)=-\int_X u\sigma_k({D^{1,1}u})dx+\oint_{M} uQ_k(u)d\mu.
\end{equation*} 
Given $f\in C^\infty(M)$ and $k\in\mathbb{N}$, define
\begin{equation*}
\mathcal{Q}_k(f) := Q_k(u_f)
\end{equation*}
for $u_f$ the solution to \eqref{eqn:degenerate_dirichlet_problem}. It follows from \cite{Kr87, Kr89a, Kr89b, Kr90, Zhou13} that $\mathcal{Q}_k$ is well-defined; it should be regarded as a fully nonlinear analogue of the Dirichlet-to-Neumann map $f\rightarrow (u_f)_{\nu}$. In terms of this operator, Theorem \ref{thm:main_thm} states that
\begin{equation}\label{eqn:trace_inequality}
	\mathcal{E}_k(u)\geq \oint_{M} f\mathcal{Q}_k(f)d\mu
\end{equation}
for all $u\in \oo{C_{f,k}}$, with equality if and only if $u=u_f$. Equation \eqref{eqn:trace_inequality} gives a trace inequality which can be regarded as a norm computation for part of the trace embedding $W^{\frac{2k}{k+1},k+1}(X)\subset W^{\frac{2k-1}{k+1},k+1}(M)$.

\begin{rmk}
We conclude this introduction with a few additional comments on the boundary operators $B_k$ of Proposition \ref{main prop}. The conditions of Proposition \ref{main prop} do not uniquely determine the boundary operators $B_k$ of Proposition \ref{main prop}. So neither the conditions in Theorem \ref{thm:main_thm} determine the functional $\mathcal{E}_{k+1}$ uniquely. Indeed, the operators are not unique even if we require additionally that the operators $B_k$ commute with diffeomorphisms, as do the operators constructed in the proof of Proposition \ref{main prop}.  A trivial source of nonuniqueness comes from the freedom to add symmetric zeroth-order terms to $B_k$.  For example, if $B_k$ satisfies the conclusions of Proposition \ref{main prop}, so does the operator
\begin{equation*}
	(w^1,\cdots,w^k) \mapsto B_k(w^1,\cdots,w^k) + cH_bw^1\dotsm w^k
\end{equation*}
for any $c\in\mathbb{R}$. More generally, one may add to the boundary operators $B_k$ any symmetric multilinear operator which is also symmetric upon pairing with integration.  For example, consider the operator $D\colon\left(C^2(\overline{X})\right)^2\to C^\infty(M)$ defined by
\begin{equation*}
	D(v,w) = \oo{\delta}\left( L(\oo{\nabla}(vw))\right) - L(\oo{\nabla} v,\oo{\nabla} w),
\end{equation*}
where $\oo{\nabla}$ is the Levi-Civita connection on $M$ and $\oo{\delta}$ is the divergence operator. It is readily verified that $(u,v,w)\mapsto\oint u\,D(v,w)d\mu$ is a symmetric trilinear form, and thus $D$ can be added to the operator $B_2$ to yield another operator $\tilde B_2$ which satisfies the conclusions of Proposition \ref{main prop}.
\end{rmk}

This article is organized as follows. In Section \ref{sec:bg} we collect some useful facts involving the complex $k$-Hessian and the CR structure on $M$. In Section \ref{sec: dimension 2} we shall first prove Proposition \ref{main prop} for $k=2$, since things are much simpler in this case, yet it still provides the essential insights to this problem. In Section \ref{sec: dimension k} we prove Proposition \ref{main prop} for any $k$ by explicitly constructing a suitable boundary operator. In Section \ref{sec:proof of main theorem} we prove Theorem \ref{thm:main_thm}.

\section{Preliminaries}
\label{sec:bg}
In this note, we will use the Greek letters $\alpha, \beta, \gamma, \cdots$ to denote indices ranging between $1$ and $n-1$ and use the Roman letters $i, j, k, \cdots$ to denote the indices ranging between $1$ and $n$. In order to avoid tedium, we will always adopt the Einstein summation convention. Let us review some background materials.
 
\subsection{The CR structure and the Kohn Laplacian}
In this section, we give a brief review of the CR structure and the Kohn Lalacian on the real hypersurface in $\mathbb{C}^n$. For more details, we refer the readers to \cite{Boggess,Wang}. 

Let $X\subset \mathbb{C}^n$ be a bounded domain with smooth boundary $M=\partial X$. The boundary $M$ has the induced metric from $\mathbb{C}^n$ and its unit outward normal vector is denoted by $\nu$. Then $T:=J\nu$ is a unit tangent vector field along $M$, where $J$ is the complex structure on $\mathbb{C}^n$. 

We denote by $(\cdot, \cdot)$ the standard Euclidean metric on $\mathbb{C}^n$. The distribution $\mathcal{H}=\{Y\in TM: \left(Y,T\right)=0\}$ is invariant under the complex structure $J$. Therefore, we can decompose $\mathcal{H}\otimes\mathbb{C}$ into the direct sum of the $\sqrt{-1}$ and $-\sqrt{-1}$ eigenspaces of $J$, which are denoted by $T^{1,0}M$ and $T^{0,1}M$ respectively. We have
\begin{align*}
T^{1,0}M=\{Y-\sqrt{-1}JY: Y\in \mathcal{H}\}, \quad T^{0,1}M=\oo{T^{1,0}M}.
\end{align*}

In doing computations, we will use a local unitary frame $\{Z_i: 1\leq i\leq n\}$ for $T^{1,0}\mathbb{C}^n$ and its dual frame $\{\theta^i: 1\leq i\leq n\}$. Let $\nab$ and $\oo{\nab}$ be the Levi-Civita connection of $\mathbb{C}^n$ and $M$ respectively. For any $u\in C^{\infty}(\oo{X})$, we write the covariant derivatives as
\begin{align*}
u_i=Z_iu, \quad u_{\bar{i}}=\oo{Z_i}u, \quad u_{i\bar{j}}=\nabla^2u\left(\oo{Z_j},Z_i\right)=\oo{Z_j}Z_iu-\nab_{\oo{Z_j}}Z_iu, \quad \mbox{etc}.
\end{align*}
As $\mathbb{C}^n$ is flat, it follows immediately that
\begin{align}\label{commutative derivatives}
u_{i\bar{j}}=u_{\bar{j}i},\quad u_{i\bar{j}\bar{k}}=u_{i\bar{k}\bar{j}}, \quad u_{i\bar{j}k}=u_{ik\bar{j}}.
\end{align}
Set
\begin{equation*}
	\oo{\nab}u:=u_{\alpha}\oo{Z_{\alpha}}+u_{\bar{\alpha}}Z_{\alpha}+(Tu)T,
\end{equation*} 
which is the gradient vector field of $u$ on $M$. We denote its holomorphic part and antiholomorphic part respectively by
\begin{equation}
	\oo{\nab}^{1,0}u:=u_{\bar{\alpha}}Z_{\alpha}, \qquad \oo{\nab}^{0,1}u:=u_{\alpha}\oo{Z_{\alpha}}. 
\end{equation}
In other word, $\oo{\nab}^{1,0}u$ and $\oo{\nab}^{0,1}u$ are respectively the projection of $\oo{\nab}u$ to $T^{1,0}M$ and $T^{0,1}M$.
And we use $\Delta$ to denote the complex Laplacian, i.e., $\Delta u=u_{i\bar{i}}$, which is half of the real Laplacian.

Along $M$, we can assume that $Z_n=\frac{1}{\sqrt{2}}\left(\nu-\sqrt{-1}T\right)$ and $\{Z_{\alpha}: 1\leq \alpha\leq n-1\}$ is a local unitary frame of $T^{1,0}M$. If we denote $X_i=\sqrt{2}\Real Z_i$ and $Y_i=-\sqrt{2}\Ima Z_i$, then $\{X_i, Y_i: 1\leq i, j\leq n\}$ is an orthonormal basis of $T\mathbb{C}^n$. In particular, $X_n=\nu$ and $Y_n=T$. The shape operator $A: TM\rightarrow TM$ and the second fundamental form are defined in the usual way: for $X, Y\in TM$, 
\begin{align*}
AX=\nab_X\nu, \quad L \left(X, Y\right)=\left(AX, Y\right).
\end{align*}
The Hermitian mean curvature $H_b$ introduced in \cite{Wang} is defined by
\begin{equation*}
H_b=H-L \left(J\nu, J\nu\right),
\end{equation*}
where $H$ is the mean curvature.

Let $\dbar_b$ be the tangential Cauchy-Riemann operator on $M$. In terms of the local frame,
\begin{align*}
\dbar_bu=u_{\bar{\alpha}}\oo{\theta^{\alpha}}.
\end{align*}
Let $\dbar_b^*$ be the adjoint of $\dbar_b$ and let $\Box_b=-\dbar_b^*\dbar_b$ be the Kohn Laplacian. The Kohn Laplacian on functions is in general not a real operator and its real part, denoted by $\Delta_b=\Real\Box_b$, is usually called the sub-Laplacian.

\subsection{The $\Gamma_k^+$-cone}
In this subsection, we describe some properties of the elementary symmetric functions and their associated convex cones.

\begin{defn} The $k$-th elementary symmetric function for
	$\lambda=(\lambda_1,\cdots,\lambda_n)\in \mathbb{R}^n$ is
	\begin{equation*}
		\sigma_{k}(\lambda):=\sum_{i_1<\cdots<i_k}\lambda_{i_1}\cdots \lambda_{i_k}.
	\end{equation*}
\end{defn}

The elementary symmetric functions are special cases of hyperbolic polynomials \cite{Garding}.  As such, they enjoy many nice properties in their associated positive cones.

\begin{defn}\label{cone}
	The positive $k$-cone is the connected component of $\left\{\lambda : \sigma_k(\lambda)>0\right\}$ which contains $(1,\cdots,1)$.  Equivalently,
	\begin{equation*}
		\Gamma_k^+=\left\{\lambda\in \mathbb{R}^n:  \sigma_{1}(\lambda)>0, \cdots, \sigma_{k}(\lambda)>0\right\}.
	\end{equation*}
\end{defn}

For example, the positive $n$-cone is
\begin{equation*}
	\Gamma_n^+ = \left\{ \lambda\in\mathbb{R}^n :  \lambda_1,\cdots,\lambda_n>0 \right\},
\end{equation*}
and the positive $1$-cone is the half-space
\begin{equation*}
	\Gamma_1^+ = \left\{ \lambda\in\mathbb{R}^n :  \lambda_1+\cdots+\lambda_n>0 \right\}.
\end{equation*}
Note that $\Gamma_k^+$ is an open convex cone and satisfies that
\begin{equation*}
	\Gamma_{n}^+ \subset \Gamma_{n-1}^+\subset \cdots \subset\Gamma_{1}^+.
\end{equation*}

Applying G{\aa}rding's theory of hyperbolic polynomials \cite{Garding}, one concludes that
$\sigma_{k}^{\frac{1}{k}}$ is a concave function in $\Gamma_{k}^{+}$.

Let $A$ be a Hermitian matrix, i.e., $\oo{A}^{\intercal}=A$. Then all its eigenvalues are real and we can similarly define a positive $k$-cone of the Hermitian matrices. 
\begin{defn}
	A Hermitian matrix $A$ is in the $\tilde{\Gamma}_{k}^+$ cone if its 
	eigenvalues satisfy that
	\begin{equation*}
		\lambda(A):=(\lambda_1(A),\cdots,\lambda_n(A))\in  \Gamma_{k}^+.
	\end{equation*}
\end{defn}
Suppose $f$ is a function on $\Gamma_{k}^+$. Denote by $F=f(\lambda(A))$ the function on $\tilde{\Gamma}_{k}^+$ induced by $f$.  It is known in \cite{Caff4} that if $f$ is concave in $\Gamma_{k}^{+}$, then the induced function $F$ is concave in $\tilde{\Gamma}_{k}^{+}$.  For this reason, we shall denote $\tilde{\Gamma}_k^+$ by $\Gamma_k^+$ and $\sigma_k(\lambda(A))$ by $\sigma_k(A)$ when there is no possibility of confusion.

For an $n\times n$ Hermitian matrix $A$, let $A_{i\bar{j}}$ be its $(i, j)$ entry. An equivalent definition of $\sigma_k(A)$ is
\begin{equation}
\sigma_{k}(A):=\frac{1}{k!}\delta^{i_1\cdots i_k}_{j_1\cdots j_k}A_{i_1\bar{j_1}} \cdots A_{i_k\bar{j_k}},
\end{equation}
where $\delta^{i_1\cdots i_k}_{j_1\cdots j_k} $ is the generalized Kronecker delta, that is to say, it is zero if $\{i_1,\cdots,i_k\}\neq\{j_1,\cdots,j_k\}$ and equals $1$ (resp.  $-1$) if $(i_1,\cdots,i_k)$ and $(j_1,\cdots,j_k)$ differ by an even (resp. odd) permutation. In particular, when $k=n$, 
\begin{equation}
\sigma_n(A)=\frac{1}{n!}\delta^{i_1\cdots i_n}_{j_1\cdots j_n}A_{i_1\bar{j_1}} \cdots A_{i_n\bar{j_n}}=\det A.
\end{equation}

The Newton transformation tensor is defined as
\begin{equation}\label{def: Newton transformation tensor}
	T_{k}(A)_{j\bar{i}}:=\frac{1}{k!}
	\delta^{ii_1\cdots i_k}_{jj_1\cdots j_k}{A}_{i_1\bar{j_1}} \cdots {A}_{i_k\bar{j_k}}.
\end{equation}
In fact, it is the linearized operator of $\sigma_{k+1}$:
\begin{equation}
	T_k(A)_{j\bar{i}}=\frac{\partial\sigma_{k+1}(A)}{\partial A_{i\bar{j}}}.
\end{equation}

\begin{defn}\label{sigmak polarization}
	The polarization of $\sigma_k$ is defined by
	\begin{equation}
		\sigma_{k}(A_1,\cdots,A_k):=\frac{1}{k!}\delta^{i_1\cdots i_k}_{j_1\cdots j_k}{(A_1)}_{i_1\bar{j_1}} \cdots {(A_k)}_{i_k\bar{j_k}}.
	\end{equation}
\end{defn}

It is called the polarization of $\sigma_k$ because $\sigma_k(A_1,\cdots,A_k)$ is the symmetric multilinear form such that $\sigma_k(A)=\sigma_k(A,\cdots,A)$.

\begin{defn}\label{Tk polarization}
	The polarized Newton transformation tensor is
	\begin{equation}
		T_{k}(A_1,\cdots,A_k)_{j\bar{i}}:=\frac{1}{k!}\delta^{ii_1\cdots i_k}_{jj_1\cdots j_k}{(A_1)}_{i_1\bar{j_1}} \cdots {(A_k)}_{i_k\bar{j_k}}. 
	\end{equation}
\end{defn}

When some components in the polarizations are the same, we adopt the notational conventions
\begin{align*}
\sigma_{k}(\overbrace{B,\cdots,B}^{l},C,\cdots,C) & :=\sigma_{k}(\overbrace{B,\cdots,B}^{l},\overbrace{C,\cdots,C}^{k-l}), \\
T_{k}(\overbrace{B,\cdots,B}^{l},C,\cdots,C)_{i\bar{j}} & :=T_{k}(\overbrace{B,\cdots,B}^{l},\overbrace{C,\cdots,C}^{k-l})_{i\bar{j}}.
\end{align*}

Some useful relations between the Newton transformation tensor $T_{k}$ and $\sigma_k$ are as follows. 
\begin{align}\label{basic identies for elementary polynomials}
\sigma_{k}(A)=\frac{1}{n-k}T_{k}(A)_{i\bar{i}}=\frac{1}{k} A_{i\bar{j}}T_{k-1}(A)_{j\bar{i}}.
\end{align}

Many useful algebraic inequalities for elements of $\Gamma_{k}^+$ can be deduced from G{\aa}rding's theory of hyperbolic polynomials \cite{Garding}.  For us, the important such inequality is the fact that if $(A_1)_{i\bar{j}},\cdots, (A_k)_{i\bar{j}}\in \Gamma_{k+1}^{+}$, then $T_k(A_1,\dots,A_k)_{i\bar{j}}$ is a nonnegative Hermitian matrix. 

For any $u\in C^{\infty}(\oo{X})$, let $D^{1,1}u$ be its complex Hessian matrix, i.e., $D^{1,1}u=(u_{i\bar{j}})_{1\leq i,j\leq n}$. For simplicity, we denote by $T_k(u)_{i\bar{j}}$ the Newton transformation tensor $T_k(D^{1,1}u)_{i\bar{j}}$. We will use its divergence free property in later sections.
\begin{prop}
	Given $k\in \mathbb{N}$ and $u\in C^{\infty}(\oo{X})$, the Newton transformation tensor $T_k(D^{1,1}u)_{i\bar{j}}$ satisfies
	\begin{equation}\label{T divergence free}
	T_k(u)_{i\bar{j}\bar{i}}:=\nab_{\bar{i}}T_k(u)_{i\bar{j}}=0,\quad \mbox{and } \quad  T_k(u)_{i\bar{j}j}:=\nab_{j}T_k(u)_{i\bar{j}}=0.
	\end{equation}
\end{prop}
\begin{proof}
	By the definition of the Newton transformation tensor \eqref{def: Newton transformation tensor}, it follows readily
	\begin{align*}
		T_{k}(u)_{i\bar{j}\bar{i}}=\sum_{l=1}^k\frac{1}{k!}
		\delta^{jj_1\cdots j_k}_{ii_1\cdots i_k}u_{j_1\bar{i_1}} \cdots u_{j_l\bar{i_l}\bar{i}} \cdots u_{j_k\bar{i_k}}.
	\end{align*}
	Since $u_{j_l\bar{i_l}\bar{i}}=u_{j_l\bar{i}\bar{i_l}}$ by \eqref{commutative derivatives} and the Kronecker delta $\delta^{jj_1\cdots j_k}_{ii_1\cdots i_k}$ is skew-symmetric in the index $i$ and $i_l$, each term on the right-hand side is zero. Therefore, we have $T_k(u)_{i\bar{j}\bar{i}}=0$. The second identity follows in a similar way.
\end{proof}

\section{Proof of Proposition \ref{main prop} for $k=2$}\label{sec: dimension 2}

Set 
\begin{equation}\label{S0 k=2}
	S_0(u,v,w):=-\int_{X} u_iv_{\bar{j}}T_1(w)_{j\bar{i}}+u_{\bar{j}}v_iT_1(w)_{j\bar{i}}+v_iw_{\bar{j}}T_1(u)_{j\bar{i}}+v_{\bar{j}}w_iT_1(u)_{j\bar{i}}+w_iu_{\bar{j}}T_1(v)_{j\bar{i}}+w_{\bar{j}}u_{i}T_1(v)_{j\bar{i}}.
\end{equation}

We first recall the divergence theorem in terms of the unitary frame $\{Z_i: 1\leq i\leq n\}$ with $Z_n=\frac{1}{\sqrt{2}}\left(\nu-\sqrt{-1}T\right)$.
\begin{lemma}
	We have the following integral identities for tensors $a_i \theta^i$ or $a_{\bar{i}}\oo{\theta^i}$ on $\oo{X}$.
	\begin{align}\label{divergence theorem}
	\int_{X} \nab_{\bar{i}}a_i=\frac{1}{\sqrt{2}}\oint_{M} a_n, \quad \mbox{ and } \quad 
	\int_{X} \nab_i a_{\bar{i}}=\frac{1}{\sqrt{2}}\oint_{M} a_{\bar{n}}.
	\end{align}
\end{lemma}
\begin{proof}
	Note that
	\begin{align*}
		\text{div}\left(a_i\theta^i\right)=\nab_{\bar{i}}a_i.
	\end{align*}
	By the divergence theorem, we obtain
	\begin{align*}
	\int_{X} \nab_{\bar{i}}a_i
	=\int_{X}\text{div} \left(a_i\theta^i\right)
	=\oint_{M}a_i\theta^i(\nu)=\frac{1}{\sqrt{2}}a_n.
	\end{align*}
	The last equality follows from that $Z_n=\frac{1}{\sqrt{2}}(\nu-\sqrt{-1}T)$.
	
	Similarly, we also have
	\begin{align*}
	\int_{X} \nab_{i}a_{\bar{i}}
	=\int_{X}\text{div} \left(a_{\bar{i}}\oo{\theta^i}\right)
	=\oint_{M}a_{\bar{i}}\oo{\theta^i}(\nu)=\frac{1}{\sqrt{2}}a_{\bar{n}}.
	\end{align*}
\end{proof}

\begin{proof}[Proof of Proposition \ref{main prop} for $k=2$]
Note that $S_0$ is symmetric.  Our goal is to rewrite \eqref{S0 k=2} in the desired form \eqref{Lk}. By applying integration by parts to $S_0$ and using the divergence free property \eqref{T divergence free}, we have
\begin{align*}
	S_0=&\int_{X} 2u v_{i\bar{j}}T_1(w)_{j\bar{i}}+2v w_{i\bar{j}}T_1(u)_{j\bar{i}}+2u w_{i\bar{j}}T_1(v)_{j\bar{i}}\\
	&-\frac{1}{\sqrt{2}}\oint_{M} uv_i T_1(w)_{n\bar{i}}+uv_{\bar{j}}T_1(w)_{j\bar{n}}+vw_i T_1(u)_{n\bar{i}}+vw_{\bar{j}}T_1(u)_{j\bar{n}}+uw_i T_1(v)_{n\bar{i}}+uw_{\bar{j}} T_1(v)_{j\bar{n}}.
\end{align*}
Note that by \eqref{def: Newton transformation tensor},
\begin{equation*}
	w_{i\bar{j}} T_1(u)_{j\bar{i}}=\delta^{ii_1}_{jj_1}w_{i\bar{j}}u_{i_1\bar{j_1}}=\delta^{ii_1}_{jj_1}u_{i\bar{j}}w_{i_1\bar{j_1}}=u_{i\bar{j}} T_1(w)_{j\bar{i}}.
\end{equation*}
Thus,
\begin{align*}
	\int_{X}vw_{i\bar{j}} T_1(u)_{j\bar{i}}
	=&\int_{X} vu_{i\bar{j}} T_1(w)_{j\bar{i}}\\
	=&-\int_{X}v_{\bar{j}} u_i T_1(w)_{j\bar{i}}+\frac{1}{\sqrt{2}}\oint_{M} vu_i T_1(w)_{n\bar{i}}\\
	=&\int_{X}uv_{i\bar{j}}T_1(w)_{j\bar{i}}+\frac{1}{\sqrt{2}}\oint_{M}vu_iT_1(w)_{n\bar{i}}-uv_{\bar{j}} T_1(w)_{j\bar{n}}.
\end{align*}
Similarly,
\begin{align*}
\int_{X}vw_{i\bar{j}} T_1(u)_{j\bar{i}}
=&\int_{X} vu_{i\bar{j}} T_1(w)_{j\bar{i}}\\
=&-\int_{X}v_{i} u_{\bar{j}} T_1(w)_{j\bar{i}}+\frac{1}{\sqrt{2}}\oint_{M} vu_{\bar{j}} T_1(w)_{j\bar{n}}\\
=&\int_{X}uv_{i\bar{j}}T_1(w)_{j\bar{i}}+\frac{1}{\sqrt{2}}\oint_{M}vu_{\bar{j}}T_1(w)_{j\bar{n}}-uv_i T_1(w)_{n\bar{i}}.
\end{align*}
Combining these two identities,
\begin{align*}
	2\int_{X}vw_{i\bar{j}} T_1(u)_{j\bar{i}}
	=2\int_{X}uv_{i\bar{j}}T_1(w)_{j\bar{i}}+\frac{1}{\sqrt{2}}\oint_{M}vu_iT_1(w)_{n\bar{i}}-uv_{\bar{j}} T_1(w)_{j\bar{n}}+vu_{\bar{j}}T_1(w)_{j\bar{n}}-uv_i T_1(w)_{n\bar{i}}.
\end{align*}

Plugging this back into $S_0$, we have
\begin{align*}
S_0=&\int_{X} 6u v_{i\bar{j}}T_1(w)_{j\bar{i}}
+\frac{1}{\sqrt{2}}\oint_{M}vu_iT_1(w)_{n\bar{i}}-uv_{\bar{j}} T_1(w)_{j\bar{n}}+vu_{\bar{j}}T_1(w)_{j\bar{n}}-uv_i T_1(w)_{n\bar{i}}\\
&-\frac{1}{\sqrt{2}}\oint_{M} uv_i T_1(w)_{n\bar{i}}+uv_{\bar{j}}T_1(w)_{j\bar{n}}+vw_i T_1(u)_{n\bar{i}}+vw_{\bar{j}}T_1(u)_{j\bar{n}}+uw_i T_1(v)_{n\bar{i}}+uw_{\bar{j}} T_1(v)_{j\bar{n}}.
\end{align*}

Since $S_0(u,v,w)$ is symmetric in $v$ and $w$, we also have
\begin{align*}
S_0=&\int_{X} 6u v_{i\bar{j}}T_1(w)_{j\bar{i}}
+\frac{1}{\sqrt{2}}\oint_{M}wu_iT_1(v)_{n\bar{i}}-uw_{\bar{j}} T_1(v)_{j\bar{n}}+wu_{\bar{j}}T_1(v)_{j\bar{n}}-uw_i T_1(v)_{n\bar{i}}\\
&-\frac{1}{\sqrt{2}}\oint_{M} uw_i T_1(v)_{n\bar{i}}+uw_{\bar{j}}T_1(v)_{j\bar{n}}+wv_i T_1(u)_{n\bar{i}}+wv_{\bar{j}}T_1(u)_{j\bar{n}}+uv_i T_1(w)_{n\bar{i}}+uv_{\bar{j}} T_1(w)_{j\bar{n}}.
\end{align*}
For simplicity, we will use $uv_i$ and $uv_{\bar{j}}$  to denote $uv_iT_1(w)_{n\bar{i}}$ and $uv_{\bar{j}}T_1(w)_{j\bar{n}}$ respectively when there is no ambiguity. Combining these two expressions of $S_0$, we obtain
\begin{align*}
S_0=&6\int_{X}uv_{i\bar{j}}T_1(w)_{j\bar{i}}-\frac{3}{2\sqrt{2}}\oint_{M}u(v_i+w_i+v_{\bar{j}}+w_{\bar{j}})
\\
&-\frac{1}{2\sqrt{2}}\oint_{M}vw_i+vw_{\bar{j}}+wv_i+wv_{\bar{j}}
+\frac{1}{2\sqrt{2}}\oint_{M}vu_i+vu_{\bar{j}}+wu_i+wu_{\bar{j}}.
\end{align*}
Denote the boundary integral by $P$:
\begin{align}\label{P k=2}
\begin{split}
P:=&-\frac{3}{2\sqrt{2}}\oint_{M}u(v_i+w_i+v_{\bar{j}}+w_{\bar{j}})
-\frac{1}{2\sqrt{2}}\oint_{M}vw_i+vw_{\bar{j}}+wv_i+wv_{\bar{j}}
\\
&+\frac{1}{2\sqrt{2}}\oint_{M}vu_i+vu_{\bar{j}}+wu_i+wu_{\bar{j}}.
\end{split}
\end{align}
Thus,
\begin{equation*}
	S_0=6\int_{X}uv_{i\bar{j}}T_1(w)_{j\bar{i}}+P.
\end{equation*}
Our goal is to write $P$ as the sum of a symmetric term and a boundary integral of the form $\oint_M u B(v,w)d\mu$ for some biliear differential operator $B: C^{\infty}(M)\times C^{\infty}(M)\rightarrow C^{\infty}(M)$. To this purpose, we take the symmetrization with respect to $u, v, w$ for the second integral in \eqref{P k=2}:
\begin{equation}\label{S1 k=2}
 S_1:=-\frac{1}{2\sqrt{2}}\oint_{M}vw_i+vw_{\bar{j}}+wv_i+wv_{\bar{j}}+uw_i+uw_{\bar{j}}+wu_i+wu_{\bar{j}}+vu_i+vu_{\bar{j}}+uv_i+uv_{\bar{j}}.
\end{equation}
Then
\begin{align}\label{term 1}
\begin{split}
S_0-S_1=6\int_{X}uv_{i\bar{j}}T_1(w)_{j\bar{i}}-\frac{1}{\sqrt{2}}\oint_{M}u(v_i+w_i+v_{\bar{j}}+w_{\bar{j}})
+\frac{1}{\sqrt{2}}\oint_{M}vu_i+vu_{\bar{j}}+wu_i+wu_{\bar{j}}.
\end{split}
\end{align}
We denote the second and the last integrals by $U_1$ and $Q$ respectively. That is,
\begin{align}\label{U1 Q}
\begin{split}
U_1:=&-\frac{1}{\sqrt{2}}\oint_{M}u(v_iT_1(w)_{n\bar{i}}+w_iT_1(v)_{n\bar{i}}+v_{\bar{j}}T_1(w)_{j\bar{n}}+w_{\bar{j}}T_1(v)_{j\bar{n}}).
\\
Q:=&\frac{1}{\sqrt{2}}\oint_{M} vu_iT_1(w)_{n\bar{i}}+vu_{\bar{j}}T_1(w)_{j\bar{n}}+wu_iT_1(v)_{n\bar{i}}+wu_{\bar{j}}T_1(v)_{j\bar{n}}.
\end{split}
\end{align}
Since $U_1$ is already in the form of $\int_M u B(v,w)$, we will focus on the term $Q$. In order to express $T_1(w)_{n\bar{n}}$, we recall some earlier work of Wang \cite{Wang} on the Hermitian mean curvature and sub-Laplacian operator. We re-organize and summarize some computations of Lemma 1 and Proposition 1 of \cite{Wang} and present them in Lemma \ref{Kohn Laplacian lemma} and Lemma \ref{lem: Wang 2} for the purpose of our setting.

As before let $\nu$ be the unit outward normal vector along $M$ and $T=J\nu$. In addition, $\oo{\nab}$ denotes the Levi-Civita connection on $M$, $\Box_b$ denotes the Kohn Laplacian on $M$ and $H_b$ is the Hermitian mean curvature of $M$. 
\begin{lemma}\label{Kohn Laplacian lemma}
	For any $f\in C^{\infty}(M)$, we have
	\begin{equation}\label{Kohn Laplacian}
		\Box_bf=\oo{\nab}^2f\left(Z_{\alpha}, \oo{Z_{\alpha}}\right)-\frac{\sqrt{-1}}{2}H_b\,Tf+\sqrt{-1}\,L \left(Z_{\alpha},T\right)\oo{Z_{\alpha}}f.
	\end{equation}
\end{lemma}
\begin{proof}
	For any $g\in C^{\infty}(M)$, we have
	\begin{align*}
		\bigl(\dbar_b f, \oo{\dbar_b g}\bigr)
		=\oint_{M}f_{\bar{\alpha}}\oo{g_{\bar{\alpha}}}
		=\oint_{M}Z_{\alpha}\left(f_{\bar{\alpha}}\oo{g}\right)-\left(Z_{\alpha}f_{\bar{\alpha}}\right)\,\oo{g}.
	\end{align*}
	Let $V=f_{\bar{\alpha}}\oo{g}Z_{\alpha}:=V^{\alpha}Z_{\alpha}$ and we compute its divergence on $M$. 
	\begin{align*}
		\text{div}^{M}\left(V\right)
		=&\left(\oo{\nab}_{\beta}V, \oo{Z_{\beta}}\right)+\left(\oo{\nab}_{T}V,T\right)
		\\
		=&Z_{\beta}\left(f_{\bar{\beta}}\bar{g}\right)-\left(V, \oo{\nab}_{\beta}\oo{Z_{\beta}}\right)-\left(V,\nab_TT\right)
		\\
		=&Z_{\beta}\left(f_{\bar{\beta}}\bar{g}\right)-\left(Z_{\alpha}, \oo{\nab}_{\beta}\oo{Z_{\beta}}\right)f_{\bar{\alpha}}\bar{g}-\left(JV,\nab_TJT\right)
		\\
		=&Z_{\beta}\left(f_{\bar{\beta}}\bar{g}\right)-\left(Z_{\alpha}, \oo{\nab}_{\beta}\oo{Z_{\beta}}\right)f_{\bar{\alpha}}\bar{g}+\sqrt{-1}\left(V,\nab_T\nu\right)
		\\
		=&Z_{\beta}\left(f_{\bar{\beta}}\bar{g}\right)-\left(Z_{\alpha}, \oo{\nab}_{\beta}\oo{Z_{\beta}}\right)f_{\bar{\alpha}}\bar{g}+\sqrt{-1}\,L \left(V,T\right)
		\\
		=&Z_{\beta}\left(f_{\bar{\beta}}\bar{g}\right)-\left(Z_{\alpha}, \oo{\nab}_{\beta}\oo{Z_{\beta}}\right)f_{\bar{\alpha}}\bar{g}+\sqrt{-1}f_{\bar{\beta}}\oo{g}\,L \left(Z_{\beta},T\right).
	\end{align*}
	Therefore,
	\begin{align*}
		-\left(\Box_bf,\oo{g}\right)=\oint_{M}\left(Z_{\alpha}, \oo{\nab}_{\beta}\oo{Z_{\beta}}\right)f_{\bar{\alpha}}\bar{g}-\sqrt{-1}f_{\bar{\alpha}}\oo{g}\,L \left(Z_{\alpha},T\right)-\left(Z_{\alpha}f_{\bar{\alpha}}\right)\oo{g}.
	\end{align*}
	Since this identity works for any $g\in C^{\infty}(M)$, we have
	\begin{align*}
		\Box_bf
		=&Z_{\alpha}\oo{Z_{\alpha}}f-\left(Z_{\alpha}, \oo{\nab}_{\beta}\oo{Z_{\beta}}\right)\oo{Z_{\alpha}}f+\sqrt{-1}\,L \left(Z_{\alpha},T\right)\oo{Z_{\alpha}}f
		\\
		=&\oo{\nab}^2f\left(Z_{\alpha}, \oo{Z_{\alpha}}\right)+\left(\oo{\nab}_{Z_{\alpha}}\oo{Z_{\alpha}}\right)f-\left(Z_{\alpha}, \oo{\nab}_{\beta}\oo{Z_{\beta}}\right)\oo{Z_{\alpha}}f+\sqrt{-1}\,L \left(Z_{\alpha},T\right)\oo{Z_{\alpha}}f
		\\
		=&\oo{\nab}^2f\left(Z_{\alpha}, \oo{Z_{\alpha}}\right)+\left(T, \oo{\nab}_{\beta}\oo{Z_{\beta}}\right)Tf+\sqrt{-1}\,L \left(Z_{\alpha},T\right)\oo{Z_{\alpha}}f
		\\
		=&\oo{\nab}^2f\left(Z_{\alpha}, \oo{Z_{\alpha}}\right)-\sqrt{-1}\left(\nab_{\beta}\nu, \oo{Z_{\beta}}\right)Tf+\sqrt{-1}\,L \left(Z_{\alpha},T\right)\oo{Z_{\alpha}}f
		\\
		=&\oo{\nab}^2f\left(Z_{\alpha}, \oo{Z_{\alpha}}\right)-\sqrt{-1}\,L \left(Z_{\alpha},\oo{Z_{\alpha}}\right)Tf+\sqrt{-1}\,L \left(Z_{\alpha},T\right)\oo{Z_{\alpha}}f.
	\end{align*}
	Recall that $X_i= \sqrt{2}\,\text{Re } Z_i$ and $Y_i=-\sqrt{2}\,\text{Im }Z_i$ for $1\leq i\leq n$. Thus,
	\begin{align*}
		L \left(Z_{\alpha},\oo{Z_{\alpha}}\right)=\frac{1}{2}L \left(X_{\alpha},X_{\alpha}\right)+\frac{1}{2}L \left(Y_{\alpha},Y_{\alpha}\right)=\frac{1}{2}H_b.
	\end{align*}
	Therefore, 
	\begin{align*}
		\Box_bf=\oo{\nab}^2f\left(Z_{\alpha}, \oo{Z_{\alpha}}\right)-\frac{\sqrt{-1}}{2}H_b\,Tf+\sqrt{-1}\,L \left(Z_{\alpha},T\right)\oo{Z_{\alpha}}f.
	\end{align*}
\end{proof}

\begin{rmk}\label{rmk integration by parts on M}
	By performing the same computation as in Lemma \ref{Kohn Laplacian lemma}, we can in fact prove the following integration by parts identity on $M$. For any smooth differential forms $a_{\alpha}\theta^{\alpha}$ or $a_{\bar{\alpha}}\oo{\theta^{\alpha}}$ on $M$ and any $f\in C^{\infty}(M)$, we have
	\begin{align*}
	\oint_{M} f_{\alpha}a_{\bar{\alpha}}
	=&\oint_{M} f\,\left(Z_{\alpha}, \oo{\nab}_{\beta}\oo{Z_{\beta}}\right)a_{\bar{\alpha}}-\sqrt{-1}f\,L \left(Z_{\alpha},T\right)a_{\bar{\alpha}}-fZ_{\alpha}a_{\bar{\alpha}},
	\\
	\oint_{M} f_{\bar{\alpha}}a_{\alpha}
	=&\oint_{M} f\left(\oo{Z_{\alpha}}, \oo{\nab}_{\bar{\beta}}Z_{\beta}\right)a_{\alpha}+\sqrt{-1}f\,L \left(\oo{Z_{\alpha}},T\right)a_{\alpha}-f\oo{Z_{\alpha}}a_{\alpha}.
	\end{align*}
	These integration by parts formulas yield that for any boundary integral with the holomorphic or antiholomorphic derivatives on $f$, we can always write it into a boundary integral whose integrand factors through $f$. 
\end{rmk}

We compare complex the Laplacian $\Delta$ on $\mathbb{C}^n$ and the Kohn Laplacian $\Box_b$ on $M$ for latter purposes.
\begin{lemma}\label{lem: Wang 2}
	For any $u\in C^{\infty}(\oo{X})$, we have
	\begin{equation}\label{Laplacian relation}
	\Delta u-u_{n\bar{n}}=\Box_bu-\sqrt{-1}u_{\bar{\alpha}}L \left(Z_{\alpha},T\right)+\frac{1}{\sqrt{2}}H_bu_{\bar{n}}
	\, \mbox{ on } M.
	\end{equation}
\end{lemma}

\begin{proof}
	Recall that $X_i= \sqrt{2}\,\text{Re } Z_i$ and $Y_i=-\sqrt{2}\,\text{Im }Z_i$ for $1\leq i\leq n$. In particular, $X_n=\nu$ and $Y_n=T=J\nu$. 
	By the definition of the Hessian matrix, 
	\begin{align*}
	\nab^2u\left(X_{\alpha}, X_{\beta}\right)
	=&X_{\alpha}X_{\beta}u-\left(\nab_{X_{\alpha}}X_{\beta}\right) u,
	\\
	\oo{\nab}^2u\left(X_{\alpha}, X_{\beta}\right)
	=&X_{\alpha}X_{\beta}u-\left(\oo{\nab}_{X_{\alpha}}X_{\beta}\right) u.
	\end{align*}
	Therefore,
	\begin{align*}
		\nab^2u\left(X_{\alpha}, X_{\beta}\right)-\oo{\nab}^2u\left(X_{\alpha}, X_{\beta}\right)
		=&-\left(\nab_{X_{\alpha}}X_{\beta}\right) u+\left(\oo{\nab}_{X_{\alpha}}X_{\beta}\right) u
		\\
		=&L \left(X_{\alpha},X_{\beta}\right)u_{\nu}.
	\end{align*}
	Similarly, 
	\begin{align*}
		\nab^2u\left(Y_{\alpha}, Y_{\beta}\right)-\oo{\nab}^2u\left(Y_{\alpha}, Y_{\beta}\right)
		=&-\left(\nab_{Y_{\alpha}}Y_{\beta}\right) u+\left(\oo{\nab}_{Y_{\alpha}}Y_{\beta}\right) u
		\\
		=&L \left(Y_{\alpha},Y_{\beta}\right)u_{\nu}.
	\end{align*}
	Combining these two identities and taking the trace, we have
	\begin{align*}
	2\Delta u-\nab^2u\left(X_n, X_n\right)-\nab^2u\left(Y_n,Y_n\right)
	=&\oo{\nab}^2u\left(X_{\alpha},X_{\alpha}\right)+\oo{\nab}^2u\left(Y_{\alpha},Y_{\alpha}\right)+H_bu_{\nu}
	\\
	=&2\oo{\nab}^2u\left(Z_{\alpha},\oo{Z_{\alpha}}\right)+H_bu_{\nu},
	\end{align*}
	where $H_b=H-L (Y_n,Y_n)$ is the Hermitian mean curvature. By \eqref{Kohn Laplacian}, it follows that
	\begin{align*}
		2\Delta u-\nab^2u\left(X_n, X_n\right)-\nab^2u\left(Y_n,Y_n\right)
		=2\Box_bu+\sqrt{-1}H_bTu-2\sqrt{-1}u_{\bar{\alpha}}L \left(Z_{\alpha},T\right)+H_bu_{\nu}
	\end{align*}
	
	Note that 
	\begin{align*}
	\nab^2u(Z_n,\oo{Z_n})=\frac{1}{2}\nab^2u\left(X_n,X_n\right)+\frac{1}{2}\nab^2u\left(Y_n,Y_n\right).
	\end{align*}
	So 
	\begin{align*}
	2\Delta u-2u_{n\bar{n}}=&2\Box_bu-2\sqrt{-1}u_{\bar{\alpha}}L \left(Z_{\alpha},T\right)+\sqrt{-1}H_bTu+H_bu_{\nu}
	\\
	=&2\Box_bu-2\sqrt{-1}u_{\bar{\alpha}}L \left(Z_{\alpha},T\right)+\sqrt{2}H_bu_{\bar{n}}.
	\end{align*}
\end{proof}

\begin{rmk}
	By the definition of the Newton transformation tensor \eqref{def: Newton transformation tensor}, we have
	\begin{equation*}
		T_1(u)_{n\bar{n}}=\delta^{ni_1}_{nj_1}u_{i_1\bar{j_1}}=\Delta u-u_{n\bar{n}}.
	\end{equation*}
	Combining this with \eqref{Laplacian relation} yields
	\begin{equation*}
		T_1(u)_{n\bar{n}}= \Box_bu-\sqrt{-1}u_{\bar{\alpha}}L \left(Z_{\alpha},T\right)+\frac{1}{\sqrt{2}}H_bu_{\bar{n}}.
	\end{equation*}
	By taking the real parts of both sides, we have
	\begin{equation}\label{Tnn}
		T_1(u)_{n\bar{n}}= \Delta_bu-\sqrt{-1}u_{\bar{\alpha}}L \left(Z_{\alpha},T\right)+\sqrt{-1}u_{\alpha}L \left(\oo{Z_{\alpha}},T\right)+\frac{1}{2}H_bu_{\nu},
	\end{equation}
	where $\Delta_b$ is the sub-Lapacian on $M$.
\end{rmk}

Now we are ready to proceed by \eqref{Tnn}. First, we decompose $Q$ in \eqref{U1 Q} as
\begin{align}\label{term 2}
	\begin{split}
	Q&=\frac{1}{\sqrt{2}}\oint_{M} vu_{\alpha}T_1(w)_{n\bar{\alpha}}+vu_{\bar{\beta}}T_1(w)_{\beta\bar{n}}+wu_{\alpha}T_1(v)_{n\bar{\alpha}}+wu_{\bar{\beta}}T_1(v)_{\beta\bar{n}}
	\\
	&\quad+\frac{1}{\sqrt{2}}\oint_{M} vu_{n}T_1(w)_{n\bar{n}}+vu_{\bar{n}}T_1(w)_{n\bar{n}}+wu_nT_1(v)_{n\bar{n}}+wu_{\bar{n}}T_1(v)_{n\bar{n}}
	\\
	&:=U_2+Q_1.
	\end{split}
\end{align}

By Remark \ref{rmk integration by parts on M}, $U_2$ can be written into the form of $\oint_{M}u\cdot B(v,w)$ for some bilinear differential operator $B$. Now we consider the integral $Q_1$. Since $u_n+u_{\bar{n}}=\sqrt{2}u_{\nu}$, $Q_1$ simplifies into
\begin{equation*}
	Q_1=\oint_{M} vu_{\nu}T_1(w)_{n\bar{n}}+wu_{\nu}T_1(v)_{n\bar{n}}.
\end{equation*}
Take the symmetrization of $Q_1$:
\begin{align}\label{S2 k=2}
S_2:=\oint_{M} vu_{\nu}T_1(w)_{n\bar{n}}+wu_{\nu}T_1(v)_{n\bar{n}}+ uv_{\nu}T_1(w)_{n\bar{n}}+wv_{\nu}T_1(u)_{n\bar{n}}+ vw_{\nu}T_1(u)_{n\bar{n}}+uw_{\nu}T_1(v)_{n\bar{n}}.
\end{align}
Thus,
\begin{align*}
	Q_1-S_2
	=-\oint_M uv_{\nu}T_1(w)_{n\bar{n}}+uw_{\nu}T_1(v)_{n\bar{n}}
	-\oint_M wv_{\nu}T_1(u)_{n\bar{n}}+vw_{\nu}T_1(u)_{n\bar{n}}.
\end{align*}
We set
\begin{equation*}
	U_3:=-\oint_M uv_{\nu}T_1(w)_{n\bar{n}}+uw_{\nu}T_1(v)_{n\bar{n}},
\end{equation*}
which is already in the form of $\oint_M u B(v,w)$. By \eqref{Tnn}, we have
\begin{align*}
Q_1-S_2-U_3=&-\oint_{M}wv_{\nu}\Delta_b u-\sqrt{-1}wv_{\nu}\bigl(u_{\bar{\alpha}}L \left(Z_{\alpha},T\right)-u_{\alpha}L \left(\oo{Z_{\alpha}},T\right)\bigr)+\frac{1}{2}wv_{\nu}H_bu_{\nu}\\
&-\oint_{M}vw_{\nu}\Delta_b u-\sqrt{-1}vw_{\nu}\bigl(u_{\bar{\alpha}}L \left(Z_{\alpha},T\right)-u_{\alpha}L \left(\oo{Z_{\alpha}},T\right)\bigr)+\frac{1}{2}vw_{\nu}H_bu_{\nu}.
\end{align*}
Set
\begin{align*}
	U_4:=-\oint_{M}(wv_{\nu}+vw_{\nu})\Delta_b u-\sqrt{-1}(wv_{\nu}+vw_{\nu})\bigl(u_{\bar{\alpha}}L \left(Z_{\alpha},T\right)-u_{\alpha}L \left(\oo{Z_{\alpha}},T\right)\bigr).
\end{align*}
Then
\begin{align}\label{term 3}
	Q_1-S_2-U_3-U_4=-\frac{1}{2}\oint_Mwv_{\nu}H_bu_{\nu}+vw_{\nu}H_bu_{\nu}:=Q_2
\end{align}
Using the fact that $\Delta_b$ is self-adjoint and the integration by parts identities in Remark \ref{rmk integration by parts on M}, we notice that $U_4$ is actually in the form of $\oint_M uB(v,w)$. It remains to consider $Q_2$. Take the symmetrization of $Q_2$:
\begin{equation}\label{S3 k=2}
	S_3=-\frac{1}{2}\oint_Mwv_{\nu}H_bu_{\nu}+vw_{\nu}H_bu_{\nu}+uw_{\nu}H_b v_{\nu}.
\end{equation} 
Thus,
\begin{equation}\label{term 4}
	Q_2-S_3=\frac{1}{2}\oint_Muw_{\nu}H_bv_{\nu},
\end{equation}
which is again in the form of $\oint_M uB(v,w)$. By combining \eqref{term 1}, \eqref{term 2}, \eqref{term 3} and \eqref{term 4}, we obtain
\begin{align*}
	S_0-S_1-S_2-S_3=12\int_X u\sigma_2(D^{1,1}v, D^{1,1}w)+\sum_{j=1}^4U_j+\frac{1}{2}\oint_Muw_{\nu}H_bv_{\nu}.
\end{align*} 
Since each $U_j$ for $1\leq j\leq 4$ is in the form of $\oint_MuB(v,w)$, there exists some bilinear differential operator $B_2(\cdot, \cdot)$ such that
\begin{align*}
S_0-S_1-S_2-S_3=12\int_{X}u\sigma_2(D^{1,1}v, D^{1,1}w)-\oint_{M}uB_2(v,w).
\end{align*}
The result follows immediately by setting
\begin{equation*}
-L_3(u,v,w):=\frac{1}{12}\bigl(S_0-S_1-S_2-S_3\bigr).
\end{equation*}
\end{proof}

Now we set the functional $\mathcal{E}_3$ as
\begin{equation*}
 \mathcal{E}_3(u):=L_3(u,u,u).
\end{equation*}

By combining \eqref{term 1}, \eqref{S2 k=2} and \eqref{S3 k=2}, $\mathcal{E}_3$ writes into
\begin{align*}
\mathcal{E}_3(u)
=-\int_{X}u\sigma_2(D^{1,1}u)+\frac{1}{2}\oint_Muu_{\nu}T_1(u)_{n\bar{n}}-\frac{1}{8}\oint_M uu_{\nu}^2H_b.
\end{align*}
By using \eqref{Tnn}, we can simplify $\mathcal{E}_3$ into
\begin{align}\label{E3}
\mathcal{E}_3(u)
=-\int_{X}u\sigma_2(D^{1,1}u)
+\frac{1}{2}\oint_{M} uu_{\nu}\Delta_bu
+\frac{1}{8}\oint_{M}uu_{\nu}^2H_b
-\frac{\sqrt{-1}}{2}\oint_{M}uu_{\nu}L\bigl(\oo{\nab}^{1,0}u-\oo{\nab}^{0,1}u,T\bigr),
\end{align}
where $\oo{\nab}^{1,0}u=u_{\bar{\alpha}}Z_{\alpha}$ and $\oo{\nab}^{0,1}u=u_{\alpha} \oo{Z_{\alpha}}$.

\begin{rmk}\label{rmk domain of E_3}
	Clearly, the functional $\mathcal{E}_{3}$ depends on at most second order tangential derivatives on $M$ and at most first order transverse derivative along $\nu$. Therefore, $\mathcal{E}_{3}(u)$ is well-defined if $u\in C^2(X)\cap C^2(M)\cap C^1(\oo{X})$.
\end{rmk}

\section{Proof of Proposition \ref{main prop} for general $k$}\label{sec: dimension k}
In this section, we will construct $B_k$ for general $k$ by an inductive argument. We begin our construction with the following symmetric multilinear differential operator
\begin{align}\label{S0}
\begin{split}
S_0(u,w^1,\cdots,w^k):=& -\sum_p \Bigl[\int_X u_i w_{\bar{j}}^p T_{k-1}(D^{1,1} w^{\w p})_{j\bar{i}}dx+\int_X w^p_i u_{\bar{j}} T_{k-1}(D^{1,1} w^{\w p})_{j\bar{i}}dx\Bigr]\\
&-\sum_{p\neq q} \int_X w_i^p w_{\bar{j}}^q T_{k-1}(D^{1,1} u, D^{1,1} w^{\w p,q})_{j\bar{i}}dx.\\
\end{split}
\end{align}
where $D^{1,1}w^{\w p}$ denotes the $(k-1)$-tuple $(D^{1,1}w^1,\cdots, D^{1,1}w^{p-1}, D^{1,1}w^{p+1},\cdots, D^{1,1}w^k)$, obtained from $(D^{1,1}w^1,\cdots,D^{1,1}w^k)$ by removing the $p$-th entry $D^{1,1}w^p$, and likewise $D^{1,1} w^{\w p,q}$ denotes the $(k-2)$-tuple obtained from $(D^{1,1}w^1,\cdots,D^{1,1}w^k)$ by removing the $p$-th entry $D^{1,1}w^p$ and the $q$-th entry $D^{1,1}w^q$.  Similar notations will be used to remove more entries from the list.  We shall perform integration by parts repeatedly to rewrite \eqref{S0} as a sum of an interior and a boundary integral, both of which have integrands which factor through $u$. 

In the following, we use $S_i$ to denote terms that are symmetric in $u, w^1,\cdots, w^k$, no matter they are of interior integral or boundary integral. We use $U_i$ to denote some boundary integral of the form $\int uB(w^1,\cdots, w^k)d\mu$ for some multi-linear operator $B$. And we use $Q_i$ to denote terms that will be repeatedly decomposed in the induction, and eventually disappear when the induction terminates. 

\begin{proof}[Proof of Proposition \ref{main prop}]
	Note that $S_0$ is symmetric.  Our goal is to rewrite \eqref{S0} in the desired form \eqref{Lk}.  To that end, writing \eqref{S0} as a sum over pairs $p\neq q$ and then integrating by parts in $X$ yields
	\begin{align*}
		S_0=& -\sum_{p\neq q} \Bigl[\frac{1}{k-1}\int_X u_i w_{\bar{j}}^p T_{k-1}(D^{1,1} w^{\w p})_{j\bar{i}}dx+\frac{1}{k-1}\int_X w^p_i u_{\bar{j}} T_{k-1}(D^{1,1} w^{\w p})_{j\bar{i}}dx
		\\&\quad \qquad + \int_X w_i^p w_{\bar{j}}^q T_{k-1}(D^{1,1} u, D^{1,1} w^{\w p,q})_{j\bar{i}}dx\Bigr]\\
		&=\sum_{p\neq q}\Bigl[\frac{2}{k-1}
		\int_X u w_{i\bar{j}}^p T_{k-1}(D^{1,1} w^{\w p})_{j\bar{i}}dx+\int_X w^p u_{i\bar{j}} T_{k-1}( D^{1,1} w^{\w p})_{j\bar{i}}dx\\
		&\quad\qquad-\frac{1}{(k-1)\sqrt{2}}\oint_{M} u w_{\bar{j}}^p T_{k-1}(D^{1,1} w^{\w p})_{j\bar{n}} d\mu
		-\frac{1}{(k-1)\sqrt{2}}\oint_{M} u w_{i}^p T_{k-1}(D^{1,1} w^{\w p})_{n\bar{i}} d\mu 
		\\
		&\quad\qquad-\frac{1}{2\sqrt{2}}\oint_{M} w^p_i w^q T_{k-1}(D^{1,1} u, D^{1,1} w^{\w p,q })_{n\bar{i}} d\mu
		-\frac{1}{2\sqrt{2}}\oint_{M} w^p w_{\bar{j}}^q T_{k-1}(D^{1,1} u, D^{1,1} w^{\w p,q })_{j\bar{n}} d\mu\Bigr].
	\end{align*}
	Performing integration by parts two more times on the second term $\int_X w^p u_{i\bar{j}} T_{k-1}(D^{1,1} w^{\w p})_{j\bar{i}}dx$, we have
	\begin{align*}
		S_0=&\sum_{p\neq q}\Bigl[\frac{k+1}{k-1}
		\int_X u w_{i\bar{j}}^p T_{k-1}(D^{1,1} w^{\w p})_{j\bar{i}}dx\\
		&\qquad-\frac{k+1}{2(k-1)\sqrt{2}}\oint_{M} u w_{i}^p T_{k-1}(D^{1,1} w^{\w p})_{n\bar{i}} d\mu 
		-\frac{k+1}{2(k-1)\sqrt{2}}\oint_{M} u w_{\bar{j}}^p T_{k-1}(D^{1,1} w^{\w p})_{j\bar{n}} d\mu
		\\
		&\qquad-\frac{1}{2\sqrt{2}}\oint_{M} w^p_i w^q T_{k-1}(D^{1,1} u, D^{1,1} w^{\w p,q })_{n\bar{i}} d\mu
		-\frac{1}{2\sqrt{2}}\oint_{M} w^p w_{\bar{j}}^q T_{k-1}(D^{1,1} u, D^{1,1} w^{\w p,q })_{j\bar{n}} d\mu
		\\
		&\qquad+\frac{1}{2\sqrt{2}}\oint_{M} w^p u_i T_{k-1}( D^{1,1} w^{\w p })_{n\bar{i}} d\mu
		+\frac{1}{2\sqrt{2}}\oint_{M} w^p u_{\bar{j}} T_{k-1}( D^{1,1} w^{\w p })_{j\bar{n}} d\mu\Bigr].
	\end{align*}
	Denote the boundary integral by $P$:
	\begin{align}\label{P}
	\begin{split}
	P=&-\frac{k+1}{2\sqrt{2}}\sum_{p}\Bigl[
	\oint_{M} u w_{i}^p T_{k-1}(D^{1,1} w^{\w p})_{n\bar{i}} d\mu 
	+\oint_{M} u w_{\bar{j}}^p T_{k-1}(D^{1,1} w^{\w p})_{j\bar{n}} d\mu\Bigr]
	\\
	&+\frac{1}{2\sqrt{2}}\sum_{p\neq q}\Bigl[-\oint_{M} w^p_i w^q T_{k-1}(D^{1,1} u, D^{1,1} w^{\w p,q })_{n\bar{i}} d\mu
	-\oint_{M} w^p w_{\bar{j}}^q T_{k-1}(D^{1,1} u, D^{1,1} w^{\w p,q })_{j\bar{n}} d\mu
	\\
	&\quad\qquad\qquad+\oint_{M} w^p u_i T_{k-1}( D^{1,1} w^{\w p })_{n\bar{i}} d\mu
	+\oint_{M} w^p u_{\bar{j}} T_{k-1}( D^{1,1} w^{\w p })_{j\bar{n}} d\mu\Bigr].
	\end{split}
	\end{align}
	Thus 
	\begin{equation}\label{S0T}
		S_0 = k^2(k+1) \int_X u \sigma_{k}(D^{1,1} w^{1},\cdots, D^{1,1} w^k)dx + P.
	\end{equation}
	We aim to write $T$ as the sum of a symmetric term and a boundary integral of the form $\int uB(w^1,\cdots, w^k)d\mu$.  To that end, consider the symmetrization of the second line in \eqref{P}:
	\begin{align}\label{S1}
	\begin{split}
	S_1&:=\frac{1}{2\sqrt{2}}\sum_{p\neq q}\Bigl[ - \oint_{M} w^p_i w^q T_{k-1}(D^{1,1} u, D^{1,1} w^{\w p,q})_{n\bar{i}}d\mu- \oint_{M} w^p w_{\bar{j}}^q T_{k-1}(D^{1,1} u, D^{1,1} w^{\w p,q})_{j\bar{n}}d\mu\\
	&\qquad\qquad\quad-\frac{1}{k-1}\oint_{M} w^p u_i T_{k-1}( D^{1,1} w^{\w p})_{n\bar{i}}d\mu-\frac{1}{k-1}\oint_{M} w^p u_{\bar{j}} T_{k-1}( D^{1,1} w^{\w p})_{j\bar{n}}d\mu
	\\&\qquad\qquad\quad-\frac{1}{k-1}\oint_{M}  u w^p_i T_{k-1}( D^{1,1} w^{\w p})_{n\bar{i}}d\mu-\frac{1}{k-1}\oint_{M}  u w^p_{\bar{j}} T_{k-1}( D^{1,1} w^{\w p})_{j\bar{n}}d\mu\Bigr].\\
	\end{split}
	\end{align}
	Note that $S_1$ is symmetric with respect to $u, w^1, \cdots, w^k$.  Combining \eqref{P} and \eqref{S1} yields
	\begin{align}
	\begin{split}\label{TS1}
	P=&S_1-\frac{k}{2\sqrt{2}}\sum_{p}\Bigl[\oint_{M} u w_i^p T_{k-1}(D^{1,1} w^{\w p})_{n\bar{i}} d\mu+\oint_{M} u w_{\bar{j}}^p T_{k-1}(D^{1,1} w^{\w p})_{j\bar{n}} d\mu\Bigr]
	\\
	&+\frac{k}{2\sqrt{2}}\sum_{p}\Bigl[\oint_{M} w^p u_i T_{k-1}( D^{1,1} w^{\w p})_{n\bar{i}}d\mu+\oint_{M} w^p u_{\bar{j}} T_{k-1}( D^{1,1} w^{\w p})_{j\bar{n}}d\mu\Bigr].
	\end{split}
	\end{align}

	Define
	\begin{align*}
	U_1 & := -\frac{k}{2\sqrt{2}}\sum_{p}\Bigl[\oint_{M} u w_i^p T_{k-1}(D^{1,1} w^{\w p})_{n\bar{i}} d\mu+\oint_{M} u w_{\bar{j}}^p T_{k-1}(D^{1,1} w^{\w p})_{j\bar{n}} d\mu\Bigr], \\
	Q & := \frac{k}{2\sqrt{2}}\sum_{p}\Bigl[\oint_{M} w^p u_i T_{k-1}( D^{1,1} w^{\w p})_{n\bar{i}}d\mu+\oint_{M} w^p u_{\bar{j}} T_{k-1}( D^{1,1} w^{\w p})_{j\bar{n}}d\mu\Bigr].
	\end{align*}
	Thus,
	\begin{equation*}
		P= U_1+ S_1+ Q.
	\end{equation*}
	Note that $U_1$ is in the desired form  $\oint u B(w^1,\cdots, w^p) d\mu$.  It remains to deal with the term $Q$.  We can decompose $Q$ as
	\begin{align*}
		Q=& \frac{k}{2\sqrt{2}}\sum_{p}\Bigl[\oint_{M} w^p u_{\alpha} T_{k-1}( D^{1,1} w^{\w p})_{n\bar{\alpha}}d\mu+\oint_{M} w^p u_{\bar{\beta}} T_{k-1}( D^{1,1} w^{\w p})_{\beta\bar{n}}d\mu\Bigr]\\
		&+\frac{k}{2\sqrt{2}}\sum_{p}\Bigl[\oint_{M} w^p u_n T_{k-1}( D^{1,1} w^{\w p})_{n\bar{n}}d\mu+\oint_{M} w^p u_{\bar{n}} T_{k-1}( D^{1,1} w^{\w p})_{n\bar{n}}d\mu\Bigr].
	\end{align*}
	Recall that the Greek indices $\alpha,\beta\in\{1,\cdots,n-1\}$ denote the holomorphic tangential directions in $T^{1,0}M$ and $n$ denotes the direction $Z_n=\frac{1}{\sqrt{2}}(\nu-\sqrt{-1}J\nu) \in T^{1,0}\mathbb{C}^n$  along $M$.  By the definition of the Newton transformation tensor \eqref{def: Newton transformation tensor}, $T_{k-1}( D^{1,1} w^{\w p})_{n \bar{n}}=\sigma_{k-1}(D^{1,1} w\rvert_{\mathcal{H}}^{\w p} )$, where $\mathcal{H}$ is the distribution $\{Y\in TM: \left(Y,T\right)=0\}$ and $D^{1,1}w\rvert_{\mathcal{H}}^{\w p}$ denotes the list of the restrictions $D^{1,1}w^1\rvert_{\mathcal{H}},\cdots,D^{1,1}w^n\rvert_{\mathcal{H}}$ with the $p$-th element removed. In terms of the local frame, $D^{1,1}w\rvert_{\mathcal{H}}$ is the Hermitian matrix $(w_{\alpha\bar{\beta}})_{1\leq \alpha,\beta\leq n-1}$.
	
	We define 
	\begin{align*}
	U_2 & := \frac{k}{2\sqrt{2}}\sum_{p}\Bigl[\oint_{M} w^p u_{\alpha} T_{k-1}( D^{1,1} w^{\w p})_{n\bar{\alpha}}d\mu+\oint_{M} w^p u_{\bar{\beta}} T_{k-1}( D^{1,1} w^{\w p})_{\beta\bar{n}}d\mu\Bigr], \\
	Q_1 & := \frac{k}{2}\sum_{p}\oint_{M} w^p u_{\nu} \sigma_{k-1}( D^{1,1} w^{\w p}\rvert_{\mathcal{H}})d\mu.
	\end{align*}
	Note that
	\begin{equation*}
		u_n+u_{\bar{n}}=Z_nu+\oo{Z_n}u=\frac{1}{\sqrt{2}}\bigl(\nu-\sqrt{-1}T\bigr)u+\frac{1}{\sqrt{2}}\bigl(\nu+\sqrt{-1}T\bigr)u=\sqrt{2}u_{\nu},
	\end{equation*} 
	and thus we have
	\begin{equation}\label{QQ1}
		Q = U_2 + Q_1
	\end{equation}
	Performing integration by parts along $M$ as in Remark \ref{rmk integration by parts on M} with $f=u$, $U_2$ writes into the form of $\oint uB(w^1,\cdots,w^p)d\mu$. Therefore, we only need to consider $Q_1$.
	
	Take the symmetrization of $Q_1$:
	\begin{align*}
	S_2&:=\sum_{p\neq q}\Bigl[ \frac{k}{2(k-1)}\oint_{M} w^p u_{\nu} \sigma_{k-1}(D^{1,1} w\rvert_{\mathcal{H}}^{\w p} )d\mu + \frac{k}{2(k-1)}\oint_{M} u w^p_{\nu} \sigma_{k-1}(D^{1,1} w\rvert_{\mathcal{H}}^{\w p} )d\mu\\
	&\qquad\quad+\frac{k}{2} \oint_{M} w^p w^q_{\nu} \sigma_{k-1}(D^{1,1} u\rvert_{\mathcal{H}} , D^{1,1} w\rvert_{\mathcal{H}}^{\w p,q} ) d\mu\Bigr].\\
	\end{align*}
	Thus, $S_2$ is symmetric with respect to $u, w^1, \cdots, w^k$.  And we have
	\begin{align}
	\begin{split}
	Q_1  =S_2 - \frac{k}{2(k-1)}\sum_{p\neq q} \oint_{M} u w^p_{\nu} \sigma_{k-1}(D^{1,1} w\rvert_{\mathcal{H}}^{\w p} ) d\mu - \frac{k}{2}\sum_{p\neq q} \oint_{M} w^p w^q_{\nu} \sigma_{k-1}(D^{1,1} u\rvert_{\mathcal{H}} , D^{1,1} w\rvert_{\mathcal{H}}^{\w p,q} )d\mu .
	\end{split}
	\end{align}
	For any $v\in C^{\infty}(\oo{X})$, denote by $\oo{D}^2v$ the Hessian with respect to the Levi-Civita connection $\oo{\nab}$ on $M$ and let $\oo{D}^{1,1}v=\oo{D}^2v\rvert_{\mathcal{H}}: T^{1,0}M\times T^{0,1}M\rightarrow \mathbb{C}$ be the restriction to the distribution $\mathcal{H}$. In terms of the local frame, $\oo{D}^{1,1}v$ is the Hermitian matrix $(\oo{\nab}^2_{Z_\alpha\oo{Z_\beta}}u)_{1\leq \alpha,\beta \leq n-1}$. Moreover, we denote by $L=L(\cdot,\cdot)$ the second fundamental form of $M$. Given $v\in C^\infty(\oo{X})$, it holds that
	\begin{equation*}
	D^{1,1} v(Z,\oo{W}) = \oo{D}^{1,1} v(Z,\oo{W}) + v_{\nu}\, L(Z,\oo{W}), \quad \mbox{ for any } Z, W \in T^{1,0}M.
	\end{equation*}
	For simplicity, we write it as
	\begin{equation}\label{ambient Hessian and the boundary Hessian}
		D^{1,1}v\rvert_{\mathcal{H}}=\oo{D}^{1,1}v+v_{\nu} L\rvert_{\mathcal{H}} ,
	\end{equation}
	where $L\rvert_{\mathcal{H}}: T^{1,0}M\times T^{0,1}M\rightarrow \mathbb{C}$ is the restriction of $L$.
	Define
	\begin{align*}
	U_3 & :=-\frac{k}{2(k-1)}\sum_{p\neq q} \oint_{M} u w^p_{\nu} \sigma_{k-1}(D^{1,1} w\rvert_{\mathcal{H}}^{\w p} ) d\mu, \\
	U_4 & :=-\frac{k}{2}\sum_{p\neq q} \oint_{M} w^p w^q_{\nu} \sigma_{k-1}(\oo{D}^{1,1} u , D^{1,1} w\rvert_{\mathcal{H}}^{\w p,q} )d\mu.\\
	Q_2 & :=-\frac{k}{2}\sum_{p\neq q} \oint_{M} w^p w^q_{\nu}u_{\nu} \sigma_{k-1}(L\rvert_{\mathcal{H}}, D^{1,1} w\rvert_{\mathcal{H}}^{\w p,q} ) d\mu.
	\end{align*}
	Thus,
	\begin{equation}\label{Q1Q2}
		Q_1=S_2+U_3+U_4+Q_2.
	\end{equation}
	By using the fact 
	\begin{equation*}
		(k-1)\sigma_{k-1}(\oo{D}^{1,1} u , D^{1,1} w\rvert_{\mathcal{H}}^{\w p,q} )=\oo{\nab}^2_{\alpha\bar{\beta}}u \cdot T_{k-2}(D^{1,1} w\rvert_{\mathcal{H}}^{\w p,q}),
	\end{equation*}
	we have
	\begin{equation*}
		U_4 =-\frac{k}{2(k-1)}\sum_{p\neq q} \oint_{M} w^p w^q_{\nu} \oo{\nab}_{\alpha\bar{\beta}}u \,T_{k-2}(D^{1,1} w\rvert_{\mathcal{H}}^{\w p,q})d\mu.
	\end{equation*}
	Integrating by parts along $M$ twice, we can rewrite $U_4$ into the form of $\oint uB(w^1, \cdots, w^p)d\mu$. As $U_3$ and $U_4$ are in the correct form, now we only need to consider $Q_2$.	To that end, take the symmetrization of $Q_2$:
	\begin{align*}
	\begin{split}
	S_3&:=-\frac{k}{2} \sum_{p\neq q\neq r} \Bigl[\frac{1}{k-2}\oint_{M} w^p w^q_{\nu}u_{\nu} \sigma_{k-1}(L\rvert_{\mathcal{H}}, D^{1,1} w\rvert_{\mathcal{H}}^{\w p,q} )d\mu \\
	&\qquad\qquad\quad+\frac{1}{2!(k-2)}\oint_{M} u w^p_{\nu} w^q_{\nu} \sigma_{k-1}(L\rvert_{\mathcal{H}}, D^{1,1} w\rvert_{\mathcal{H}}^{\w p,q} ) d\mu\\
	&\qquad\qquad\quad+\frac{1}{2!}\oint_{M} w^p w^q_{\nu}w^r_{\nu} \sigma_{k-1}( L\rvert_{\mathcal{H}}, D^{1,1} u\rvert_{\mathcal{H}}, D^{1,1} w\rvert_{\mathcal{H}}^{\w p,q, r} )d\mu \Bigr].\\
	\end{split}
	\end{align*}
	Note that $S_3$ is symmetric with respect to $u,w^1,\cdots,w^k$. By defining
	\begin{align*}
	U_5 & :=\frac{k}{4(k-2)}\sum_{p\neq q\neq r}\oint_{M} u w^p_{\nu} w^q_{\nu} \sigma_{k-1}(L\rvert_{\mathcal{H}},  D^{1,1} w\rvert_{\mathcal{H}}^{\w p,q} ) d\mu,\\
	U_6 & :=\frac{k}{4}\sum_{p\neq q\neq r}\oint_{M} w^p w^q_{\nu}w^r_{\nu} \sigma_{k-1}( L\rvert_{\mathcal{H}}, \oo{D}^{1,1} u, D^{1,1} w\rvert_{\mathcal{H}}^{\w p,q,r} ) d\mu,
	\\
	Q_3 & := \frac{k}{4}\sum_{p\neq q\neq r}\oint_{M} w^p w^q_{\nu}w^r_{\nu}u_{\nu} \sigma_{k-1}( L\rvert_{\mathcal{H}}, L\rvert_{\mathcal{H}},  D^{1,1} w\rvert_{\mathcal{H}}^{\w p,q,r} )d\mu,
	\end{align*}
	we have
	\begin{equation}\label{Q2Q3}
		Q_2= S_3+U_5+U_6+Q_3.
	\end{equation}
	
	As above, integration by parts along $M$ implies that both $U_5$ and $U_6$ are of the form $\oint uB(w^1,\cdots,w^k)d\mu$. Thus, we only need to consider $Q_3$.
	
	Proceeding in this way, for all $2\leq i\leq k$ we make the following definitions.  First, define
	\begin{align}
	\begin{split}\label{Si}
	S_i & := \frac{(-1)^ik}{2} \sum_{p_1\neq \cdots \neq p_i}  \Bigl[\frac{1}{(i-2)!(k+1-i)}\oint_{M} w^{p_1}  w^{p_2}_{\nu}\cdots w^{p_{i-1}}_{\nu} u_{\nu} \\
	&\qquad\qquad\qquad\qquad \times \sigma_{k-1}(\overbrace {L\rvert_{\mathcal{H}}, \cdots, L\rvert_{\mathcal{H}}}^{i-2}, D^{1,1} w\rvert_{\mathcal{H}}^{\w p_1,\cdots, p_{i-1}} ) d\mu\\
	&\quad\qquad\qquad\qquad+\frac{1}{(i-1)!(k+1-i)}\oint_{M} u w^{p_1}_{\nu} \cdots w^{p_{i-1}}_{\nu} \sigma_{k-1}(\overbrace {L\rvert_{\mathcal{H}}, \cdots, L\rvert_{\mathcal{H}}}^{i-2}, D^{1,1} w\rvert_{\mathcal{H}}^{\w p_1,\cdots, p_{i-1}} )d\mu \\
	&\quad\qquad\qquad\qquad+\frac{1}{(i-1)!}\oint_{M} w^{p_1}  w^{p_2}_{\nu}\cdots w^{p_{i}}_{\nu}  \sigma_{k-1}( \overbrace {L\rvert_{\mathcal{H}}, \cdots, L\rvert_{\mathcal{H}}}^{i-2},D^{1,1} u\rvert_{\mathcal{H}}, D^{1,1} w\rvert_{\mathcal{H}}^{\w p_1,\cdots, p_i} ) d\mu\Bigr] .
	\end{split}
	\end{align}
	Note that $S_i$ is symmetric with respect to $u,w^1,\cdots,w^k$.  Next, define
	\begin{align*}
	U_{2i-1}& := \frac{(-1)^{i+1}k}{2(i-1)!(k+1-i)}\sum_{p_1\neq \cdots \neq p_i}\oint_{M} u w^{p_1}_{\nu} \cdots w^{p_{i-1}}_{\nu} d\mu \\
	&\quad \quad  \quad \quad \quad \quad \quad  \quad \quad \quad \quad \quad \quad  \quad \quad \times \sigma_{k-1}(\overbrace {L\rvert_{\mathcal{H}}, \cdots, L\rvert_{\mathcal{H}}}^{i-2}, D^{1,1} w\rvert_{\mathcal{H}}^{\w p_1,\cdots, p_{i-1}} ) d\mu,\\
	U_{2i}& := \frac{(-1)^{i+1}k}{2(i-1)!}\sum_{p_1\neq \cdots \neq p_i} \oint_{M} w^{p_1}  w^{p_2}_{\nu}\cdots w^{p_i}_{\nu} \sigma_{k-1}(\overbrace {L\rvert_{\mathcal{H}}, \cdots, L\rvert_{\mathcal{H}}}^{i-2},\oo{D}^{1,1}u, D^{1,1} w\rvert_{\mathcal{H}}^{\w p_1,\cdots, p_i})d\mu,\\
	Q_i&:= \frac{(-1)^{i+1} k}{2(i-1)!}\sum_{p_1\neq \cdots \neq p_i}\oint_{M} w^{p_1} w^{p_2}_{\nu}\cdots w^{p_i}_{\nu} u_{\nu} \sigma_{k-1}(\overbrace {L\rvert_{\mathcal{H}}, \cdots, L\rvert_{\mathcal{H}}}^{i-1}, D^{1,1} w\rvert_{\mathcal{H}}^{\w p_1,\cdots,p_i} ) d\mu.	
	\end{align*}
	Then we have
	\begin{equation}
		Q_{i-1}=S_i+U_{2i-1}+U_{2i}+Q_i.
	\end{equation}
	
	As above, integration by parts along $M$ implies that both $U_{2i-1}$ and $U_{2i}$ are of the form $\oint uB(w^1,\cdots,w^k)d\mu$. We can inductively perform this argument until $i=k-1$ and get 
	\begin{equation*}
		Q_{k-1}=S_{k}+U_{2k-1}+U_{2k}+Q_k,
	\end{equation*}  	
	where 
	\begin{equation*}
		Q_k:= \frac{(-1)^{k+1} k}{2(k-1)!}\sum_{p_1\neq \cdots \neq p_k}\oint_{M} w^{p_1} w^{p_2}_{\nu}\cdots w^{p_k}_{\nu} u_{\nu} \sigma_{k-1}(\overbrace {L\rvert_{\mathcal{H}}, \cdots, L\rvert_{\mathcal{H}}}^{k-1}) d\mu.
	\end{equation*}
	It remains to write $Q_k$ into as the sum of a symmetric integral and a boundary integral whose integrand factors through $u$.  To that end, we define
	\begin{align}
	\begin{split}\label{Sk+1}
	S_{k+1} :=& \frac{(-1)^{k+1}k}{2(k-1)!}\sum_{p_1\neq \cdots \neq p_k} \Bigl[\oint_{M} w^{p_1}  w^{p_2}_{\nu}\cdots w^{p_{k}}_{\nu} u_{\nu}\sigma_{k-1}(L\rvert_{\mathcal{H}})d\mu \\
	&\qquad\qquad\qquad\qquad+\frac{1}{k}\oint_{M} u w^{p_1}_{\nu} \cdots w^{p_{k}}_{\nu} \sigma_{k-1}(L\rvert_{\mathcal{H}} )d\mu \Bigr] .
	\end{split}
	\end{align}
	If we define
	\begin{equation*}
		U_{2k+1} := \frac{(-1)^{k}}{2(k-1)!}\sum_{p_1\neq \cdots \neq p_k}\oint_{M} u w^{p_1}_{\nu} \cdots w^{p_{k}}_{\nu} \sigma_{k-1}(L\rvert_{\mathcal{H}})d\mu,
	\end{equation*}
	then 
	\begin{equation}\label{QkSk}
		Q_k = S_{k+1} + U_{2k+1}.
	\end{equation}
	Note that $S_{k+1}$ is symmetric with respect to $u,w^1,\cdots,w^k$ and $U_{2k+1}$ is of the form $\oint uB(w^1,\cdots,w^k)d\mu$. Therefore, $Q_k$ is now in the desired form.
	
	In summary, we have shown that
	\begin{equation}
	\label{final induction}
	S_0 - \sum_{i=1}^{k+1} S_i = k^2(k+1)\int_X u\,\sigma_{k}(D^{1,1}w^1,\cdots,D^{1,1}w^{k}) dx+ \sum_{i=1}^{2k+1} U_i
	\end{equation}
	and observed that the left-hand side is symmetric in $u,w^1,\cdots,w^k$ while the right-hand side is of the form $\oint uB(w^1,\cdots,w^k)d\mu$.  The result therefore follows by defining
	\begin{equation}\label{LkS}
		L_{k+1}(u,w^1,w^2,\cdots, w^k):=-\frac{1}{k^2(k+1)}\Bigl(S_0 - \sum_{i=1}^{k+1} S_i \Bigr).
	\end{equation}
\end{proof}

For any $u\in C^{\infty}(\oo{X})$, we define
\begin{align*}
\mathcal{E}_{k+1}(u):=L_{k+1}(\overbrace{u, \cdots, u}^{k+1}).
\end{align*}

	We are going to show in the proof of Proposision \ref{extension of Ek} that this definition of $\mathcal{E}_{k+1}$ coincides with the expression given in \eqref{Ek}, and from there one can see $\mathcal{E}_{k+1}(u)$ does not involve second order transverse derivatives. 	

\begin{prop}\label{extension of Ek}
	The functional $\mathcal{E}_{k+1}$ depends on at most second order tangential derivatives on $M$ and at most first order transverse derivatives. In particular, we only need $u\in C^2(X)\cap C^2(M)\cap C^1(\oo{X})$ to define $\mathcal{E}_{k+1}(u)$.
\end{prop}

\begin{proof}
	Given $u\in C^{\infty}(\oo{X})$, for $0\leq i\leq k+1$ and $1\leq j\leq 2k+1$, we denote
	\begin{align*}
	S_i(u):=S_i(u,\cdots, u), \quad U_j(u):=U_j(u,\cdots, u).
	\end{align*}
	If we set $w^1=w^2=\cdots=w^k=u\in C^{\infty}(X)$ in \eqref{TS1}, then we have $P=S_1(u)$. Thus, \eqref{S0T} implies that
	\begin{equation*}
	S_0(u)=k^2(k+1)\int_X u\sigma_k(D^{1,1}u)dx+S_1(u).
	\end{equation*}
	Combining this with \eqref{LkS}, we obtain
	\begin{equation}
	\mathcal{E}_{k+1}(u)=-\int_X u\sigma_k(D^{1,1}u)dx+\frac{1}{k^2(k+1)}\sum_{i=2}^{k+1}S_i(u).
	\end{equation}
	Note that by \eqref{Si} and \eqref{Sk+1}, we have that for $2\leq i\leq k+1$,
	\begin{equation}
	S_i(u)=(-1)^i\frac{k(k+1)}{2}\binom{k}{i-1}\oint_{M}u u_{\nu}^{i-1}\sigma_{k-1}(\overbrace{L\rvert_{\mathcal{H}},\cdots,L\rvert_{\mathcal{H}}}^{i-2}, \overbrace{D^{1,1}u\rvert_{\mathcal{H}}, \cdots, D^{1,1} u}^{k+1-i}\rvert_{\mathcal{H}}).
	\end{equation}
	The result therefore follows immediately by \eqref{ambient Hessian and the boundary Hessian}.
\end{proof}

\section{Proof of Theorem \ref{thm:main_thm}}\label{sec:proof of main theorem}
It is straightforward to compute the first and second variations of the energy functional $\mathcal{E}_{k+1}$ associated to the symmetric multilinear form constructed by Proposition \ref{main prop}.

\begin{prop}
	\label{prop:first_variation}
	Let $X\subset \mathbb{C}^n$ be a bounded smooth domain with boundary $M=\partial X$.  Let $u,v\in C^\infty(\oo{X})$ and suppose that $v\vert_M=0$.  Then
	\begin{equation}
	\label{eqn:first_variation}
	\frac{d}{dt}\Big|_{t=0}\mathcal{E}_{k+1}(u+tv) = -(k+1)\int_X v\,\sigma_k(D^{1,1}u,\cdots,D^{1,1}u) dx.
	\end{equation}
\end{prop}

\begin{proof}
	Since $L_{k+1}$ is symmetric, we compute that
	\begin{equation*}
		\frac{d}{dt}\Big|_{t=0}\mathcal{E}_{k+1}(u+tv) = (k+1)L_{k+1}(v,u,\cdots,u) .
	\end{equation*}
	Since $v\vert_M=0$, we see that the boundary integral in \eqref{Lk} vanishes. The result therefore follows.
\end{proof}

\begin{prop}
	\label{prop:second_variation}
	Let $X\subset \mathbb{C}^n$ be a bounded smooth domain with boundary $M=\partial X$.  Let $u,v\in C^\infty(\oo{X})$ and suppose that $v\vert_M=0$.  Then
	\begin{equation}\label{eqn:second_variation}
		\frac{d^2}{dt^2}\Big|_{t=0}\mathcal{E}_{k+1}(u+tv) = (k+1)\int_X v_iv_{\bar{j}} T_{k-1}(D^{1,1}u)_{j\bar{i}}dx .
	\end{equation}
	In particular, if $u\in \Gamma_k^+$, then
	\begin{equation}\label{eqn:Ek convexity}
		\frac{d^2}{dt^2}\Big|_{t=0}\mathcal{E}_{k+1}(u+tv)\geq 0
	\end{equation}
	for all $v\in C^\infty(\oo{X})$ such that $v\vert_M=0$.
\end{prop}

\begin{proof}
	Since $L_{k+1}$ is symmetric, we compute that
	\begin{equation*}
		\frac{d^2}{dt^2}\Big|_{t=0} \mathcal{E}_{k+1}(u+tv) = k(k+1)L_{k+1}(v,v,u,\cdots,u).
	\end{equation*}
	Since $v\vert_M=0$, it follows that
	\begin{align*}
	\frac{d^2}{dt^2}\Big|_{t=0}\mathcal{E}_{k+1}(u+tv) & = -k(k+1)\int_X v\,\sigma_k(D^{1,1}v,D^{1,1}u,\cdots,D^{1,1}u) dx, \\
	& = -(k+1)\int_X v T_{k-1}(D^{1,1}u)_{j\bar{i}} v_{i\bar{j}}dx \\
	& = (k+1)\int_X v_iv_{\bar{j}} T_{k-1}(D^{1,1}u)_{j\bar{i}} dx.
	\end{align*}
	The last conclusion follows from the fact that if $u\in \Gamma_k^+$, then $T_{k-1}(D^{1,1}u)_{i\bar{j}}$ is nonnegative.
\end{proof}

\begin{rmk}\label{rmk extension of variations}
	By Remark \ref{rmk domain of E_3} and Proposition \ref{extension of Ek}, the functional $\mathcal{E}_{k+1}$ is actually well-defined on $C^2(X)\cap C^2(M)\cap C^1(\oo{X})$. Given $u\in \oo{C_{f,k}}$, $u$ is contained in $C^{1,1}(\oo{X})\cap C^2(M)$. We can construct $u_{\varepsilon}\in C^{\infty}(\oo{X})$ such that as $\varepsilon\rightarrow 0$, $\|u_{\varepsilon}-u\|_{C^1(\oo{X})}\rightarrow 0$ and $\|u_{\varepsilon}-u\|_{W^{2,p}(X)}\rightarrow 0$ for any $1\leq p <\infty$.  By a density argument, the variation identities \eqref{eqn:first_variation}, \eqref{eqn:second_variation} and the convexity property \eqref{eqn:Ek convexity} actually hold for $u\in \oo{C_{f,k}}$ and $v\in C^{1,1}(\oo{X})$ with $v\vert_M=0$.
\end{rmk}

We are now ready to prove Theorem \ref{thm:main_thm}.
\begin{proof}[Proof of Theorem \ref{thm:main_thm}]
	By Proposition \ref{prop:first_variation} and Remark \ref{rmk extension of variations}, the solution $u_f$ to \eqref{eqn:degenerate_dirichlet_problem} is a critical point of the functional $\mathcal{E}_{k+1}\colon \overline{C_{f,k}}\to \mathbb{R}$.  By Proposition \ref{prop:second_variation} and Remark \ref{rmk extension of variations}, the restriction $\mathcal{E}_{k+1}\colon\overline{C_{f,k}}\to \mathbb{R}$ is a convex functional.  Since $\overline{C_{f,k}}$ is convex, $u_f$ realizes the infimum of $\mathcal{E}_{k+1}\colon\overline{C_{f,k}}\to \mathbb{R}$.  Indeed, if not, then there is a $u\in\overline{C_{f,k}}$ such that $\mathcal{E}_{k+1}(u)<\mathcal{E}_{k+1}(u_f)$.  Since $\overline{C_{f,k}}$ is convex, it follows that $tu+(1-t)u_f\in\overline{C_{f,k}}$ for all $t\in[0,1]$.  Denote $\mathcal{E}_{k+1}(t):=\mathcal{E}_{k+1}(tu+(1-t)u_f)$.  Since $\mathcal{E}_{k+1}(u)<\mathcal{E}_{k+1}(u_f)$, there exists a $t^\ast\in[0,1]$ such that $\mathcal{E}_{k+1}^\prime(t^\ast)<0$.  This contradicts the facts that $\mathcal{E}_{k+1}^\prime(0)=0$ and $\mathcal{E}_{k+1}^{\prime\prime}\geq0$ for all $t\in[0,1]$.

\end{proof}

\section*{Acknowledgments}
The first author would like to thank Pengfei Guan for suggestions and interest of this work. The second author is thankful to Song-Ying Li and Xiangwen Zhang for the friendly discussions about the complex $k$-Hessian equation.
\bibliographystyle{plain}
\bibliography{references}

\end{document}